\RequirePackage{fix-cm}
\RequirePackage{amsthm}
\documentclass[sn-mathphys-num,pdflatex]{sn-jnl}

\usepackage{graphicx}
\usepackage{multirow}
\usepackage{amsmath,amssymb,amsfonts}
\usepackage{mathrsfs}
\usepackage[title]{appendix}
\usepackage{xcolor}
\usepackage{textcomp}
\usepackage{manyfoot}
\usepackage{booktabs}
\usepackage{algorithm}
\usepackage{algorithmicx}
\usepackage{algpseudocode}
\usepackage{listings}
\usepackage{anyfontsize}

\usepackage{hyperref}
\usepackage{lmodern}
\usepackage{enumerate} 
\usepackage{enumitem}

\newtheorem{theorem}{Theorem}[section]
\newtheorem{definition}[theorem]{Definition}
\newtheorem{lemma}[theorem]{Lemma}
\newtheorem{corollary}[theorem]{Corollary}
\newtheorem{example}[theorem]{Example}
\newtheorem{remark}[theorem]{Remark}

\numberwithin{equation}{section} 
\numberwithin{algorithm}{section} 

\newcounter{assumcounter}
\newtheorem{assum}[assumcounter]{Assumption}

\newcounter{oraclecounter}
\newtheorem{oracle}[oraclecounter]{Oracle}

\newcommand{\eps}{\varepsilon}
\newcommand{\R}{\mathbb{R}}
\newcommand{\N}{\mathbb{N}}
\newcommand{\T}{\mathcal{T}}
\newcommand{\calR}{\mathcal{R}}
\newcommand{\deriv}{D}
\newcommand{\epsmax}{\varepsilon_{max}}
\newcommand{\C}{\mathcal{C}}
\newcommand{\glrad}{\Delta}
\newcommand{\hx}{\hat{x}}
\newcommand{\lc}[1]{lower-$\C^{#1}$}
\newcommand{\jthr}{j_{thr}}
\newcommand{\Bcl}{\bar{B}}
\newcommand{\algApproxW}{Alg.\ \ref{algo:approx_W}}
\newcommand{\algLocal}{Alg.\ \ref{algo:local_method}}
\newcommand{\algGlobal}{Alg.\ \ref{algo:conceptual_global_method}}

\DeclareMathOperator*{\conv}{conv}
\DeclareMathOperator*{\cl}{cl}
\DeclareMathOperator*{\argmin}{arg\,min}

\newcommand{\superpolyak}{\texttt{SuperPolyak}}
\newcommand{\vubundle}{\texttt{VUbundle}}
\newcommand{\hanso}{\texttt{HANSO}}
\newcommand{\ipopt}{\texttt{IPOPT}}
\newcommand{\mexipopt}{\texttt{mexIPOPT}}
\newcommand{\epsthr}{\eps_{\text{thr}}}

\begin{document}

\title{Superlinear convergence in nonsmooth optimization via higher-order cutting-plane models}

\author*[1]{\fnm{Bennet} \sur{Gebken}}\email{bennet.gebken@tum.de}

\author[1]{\fnm{Michael} \sur{Ulbrich}}\email{mulbrich@ma.tum.de}

\affil[1]{\orgdiv{Department of Mathematics}, \orgname{Technical University of Munich}, \orgaddress{\street{Boltzmannstr. 3}, \city{Garching b. München}, \postcode{85748}, \country{Germany}}}

\abstract{
    A cutting-plane model for a nonsmooth function is the maximum of several first-order expansions centered at different points. Using such a model in a bundle method leads to linear convergence (of serious steps) to a minimum. In smooth optimization, superlinear convergence can be achieved by using higher-order models. We show that the same is true for the nonsmooth case, i.e., we show that cutting-plane models involving higher-order expansions can be used to achieve superlinear convergence in nonsmooth optimization. We first formally define higher-order cutting-plane models for \lc{2} functions and derive an error estimate. Afterwards, we construct a trust-region bundle method based on these models that achieves local superlinear convergence of serious steps, and overall superlinear convergence for certain finite max-type functions. Finally, we verify the superlinear convergence in numerical experiments.
}


\maketitle

\section{Introduction} \label{sec:introduction}

Given a function $f$ and a point $x^j$, a fundamental strategy for finding a point $x^{j+1}$ with $f(x^{j+1}) < f(x^j)$ is to generate a simple local model of $f$ around $x^j$ and then minimize the model. If this is done iteratively and the decrease per iteration is sufficient, then a sequence $(x^j)_j$ is obtained that converges to a local minimum of $f$, with a speed that depends on the accuracy of the model. When $f$ is smooth, Taylor expansion can be used for building the model. For example, first-order Taylor expansion leads to the steepest descent method, converging linearly under certain assumptions (\cite{NW2006}, Thm.\ 3.4), and a more accurate second-order Taylor expansion leads to Newton's method, typically converging (locally) quadratically (\cite{NW2006}, Thm.\ 3.5).  When $f$ is nonsmooth, Taylor expansion fails to yield a useful model (cf.\ \cite{L1989}, Sec.\ 3). In this case, a standard approach is to use a cutting-plane model, which is the piecewise linear maximum of several first-order Taylor expansions (potentially using subgradients instead of gradients) at points close to the current point. In accordance with the first-order model in the smooth case, bundle methods using this model achieve linear convergence of so-called serious steps (when assuming a subdifferential error bound, cf.\ \cite{ASS2023}). This analogy for first-order models naturally raises the question whether the maximum of higher-order Taylor expansions, i.e., a higher-order cutting-plane model, can be used as a model to achieve higher orders of convergence. We show that for \lc{2} functions \cite{RW1998} satisfying a polynomial growth assumption, this is indeed the case, yielding local R-superlinear convergence of serious steps, with an order that depends on the order of the model and the order of growth. Furthermore, stepwise R-superlinear convergence of the overall sequence is shown for finite max-type functions. 

While methods with superlinear convergence, like quasi-Newton methods, have been the state of the art in smooth optimization for a long time, this speed is significantly more difficult to (provably) achieve in nonsmooth optimization (and was even described as a ``wondrous grail'' in \cite{MS2012}). Note that we consider convergence rates in the sense of Q- or R-convergence (see, e.g., \cite{NW2006}, Appendix A.2) with respect to oracle calls, which differ from \emph{non-asymptotic} convergence rates, like the ones considered in \cite{ZLJ2020,DG2023}. Furthermore, we emphasize that we are concerned with black-box optimization, in the sense that we can evaluate the objective value and its (generalized) derivatives, but do not have access to any potential nonsmooth structure of the objective like a DC \cite{TD2018} or a composite structure \cite{KL2020}. To the best of the authors' knowledge, there are only few methods that are able to achieve superlinear convergence for such a general case: 
\begin{itemize}
    \item The $\mathcal{V}\mathcal{U}$\emph{-algorithm} \cite{MS2005,LS2020} is based on the observation that locally around the minimum of a nonsmooth function $f$, the variable space can often be decomposed into a $\mathcal{V}$-space, in which $f$ grows linearly, and a $\mathcal{U}$-space, in which $f$ behaves smoothly. If these spaces are known, then Newton-like steps can be performed along $\mathcal{U}$ to inherit the fast convergence of Newton's method. However, automatically identifying these spaces in a black-box setting is difficult and has, to the authors' knowledge, only been achieved for certain convex, piecewise differentiable functions when an active index is available \cite{DSS2009}.
    \item The method \emph{SuperPolyak} \cite{CD2024} is a modification of the bundle method by Polyak \cite{P1969}. It achieves superlinear convergence for functions with a sharp minimum (around which $f$ grows linearly) when the optimal value is available.
    \item The \emph{bundle-Newton method} \cite{LV1998} is based on using (convexified) second-order cutting-plane models together with ideas from sequential quadratic programming (SQP), but superlinear convergence can only be proven under a certain smoothness assumption on $f$.
    \item Finally, the $k$-\emph{bundle Newton method} from \cite{LW2019} also combines second-order models with SQP, resulting in a method that achieves local stepwise quadratic convergence on a certain class of well-behaved, strongly convex, piecewise differentiable functions. However, the correct choice of the (fixed) bundle size requires some additional knowledge of $f$.
\end{itemize}
(We mention that the unpublished preprint \cite{G2025} can be seen as a predecessor to the current work. It used the maximum of second-order Taylor expansions as a model, but a different algorithmic setting and a different function class meant that no results on the speed of convergence could be given.)

The idea of this work is to use the maximum of Taylor expansions of arbitrary order $q \in \N$ with centers $y \in W$ (with $W$ being the ``bundle'' in the language of bundle methods) as models for nonsmooth functions $f$. Fig.\ \ref{fig:sketch_model} shows a visualization of this idea.
\begin{figure}
    \centering
    \parbox[b]{0.32\textwidth}{
        \centering 
        \includegraphics[width=0.32\textwidth]{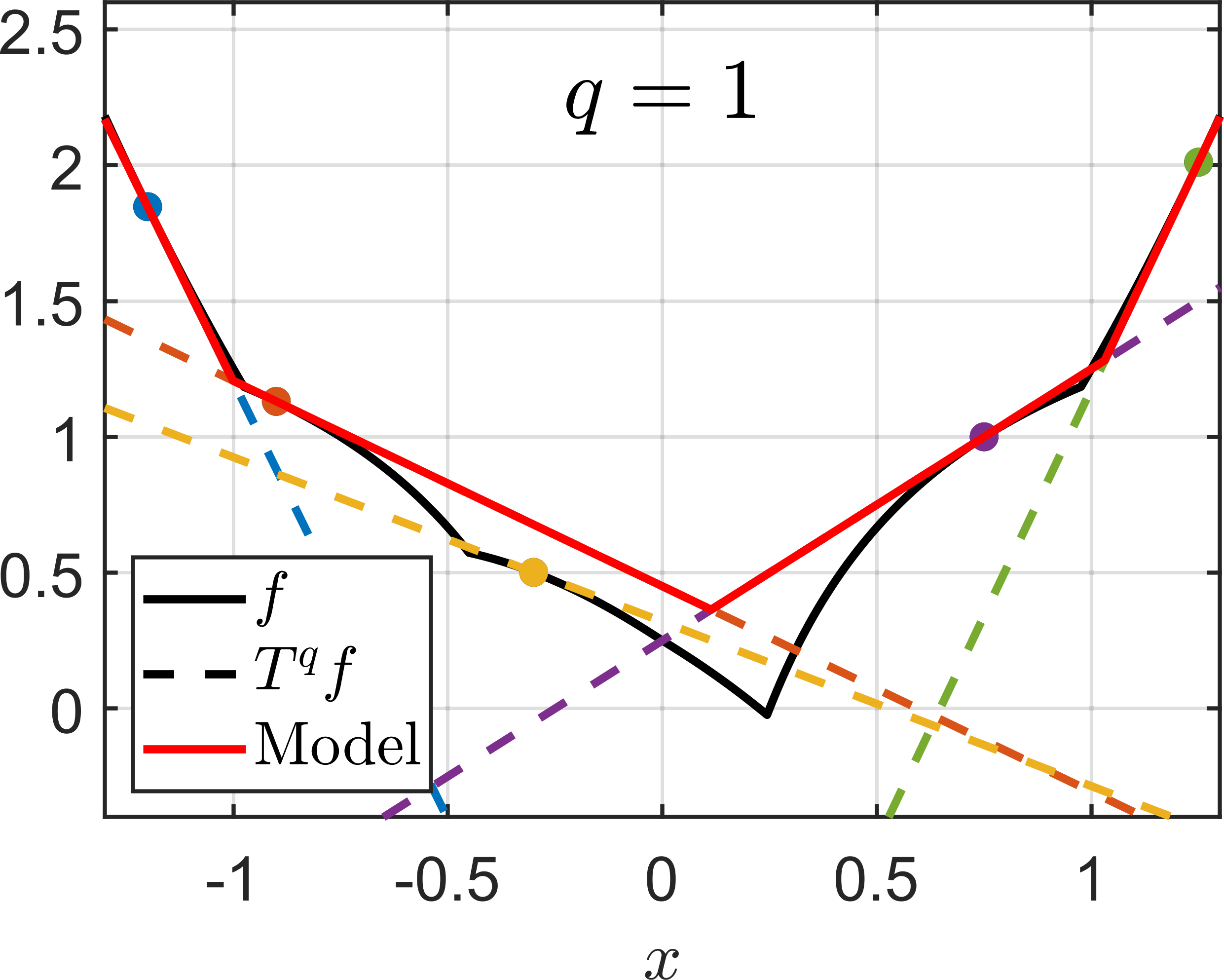}\\
    }
    \parbox[b]{0.32\textwidth}{
        \centering 
        \includegraphics[width=0.32\textwidth]{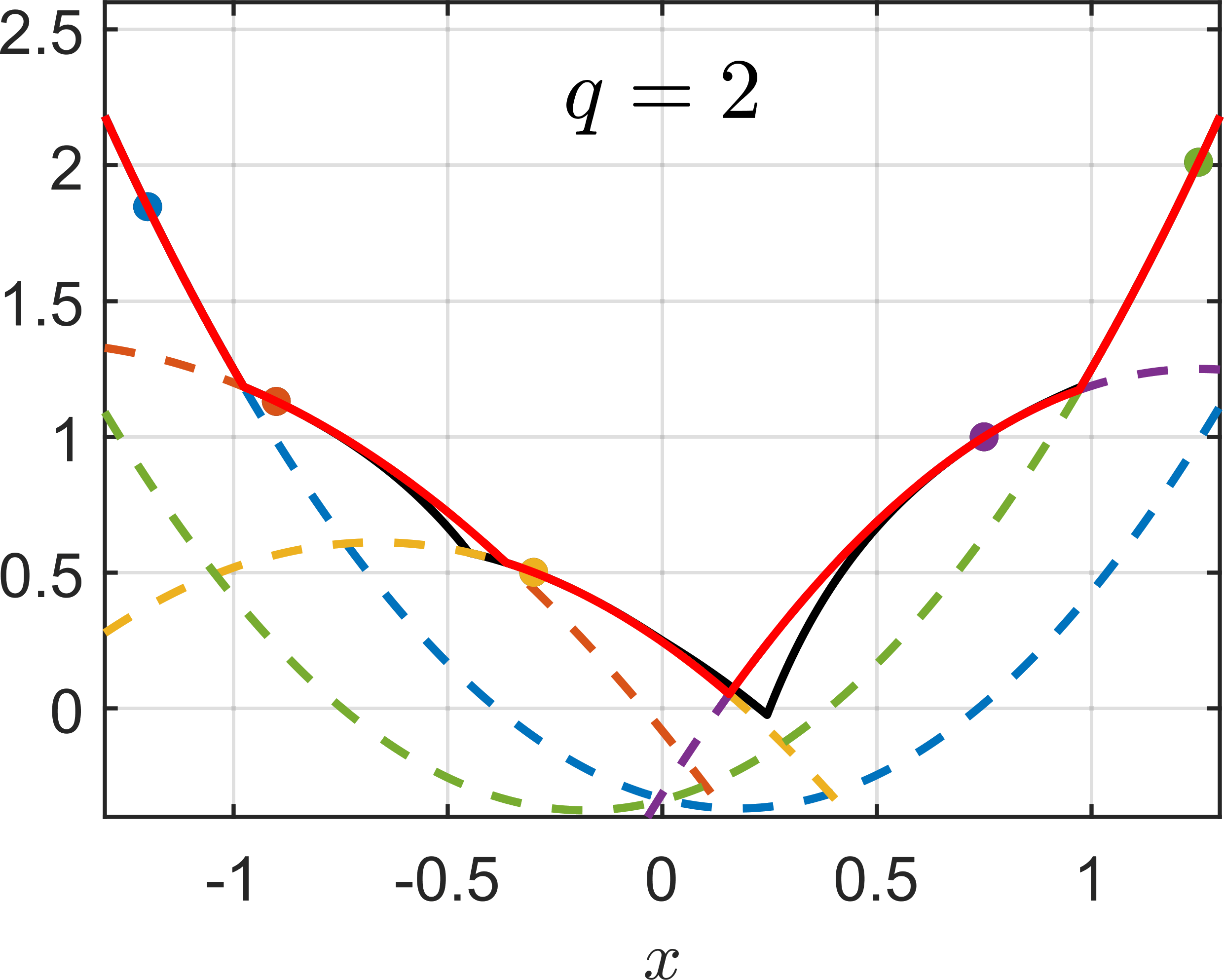}\\
    }
    \parbox[b]{0.32\textwidth}{
        \centering 
        \includegraphics[width=0.32\textwidth]{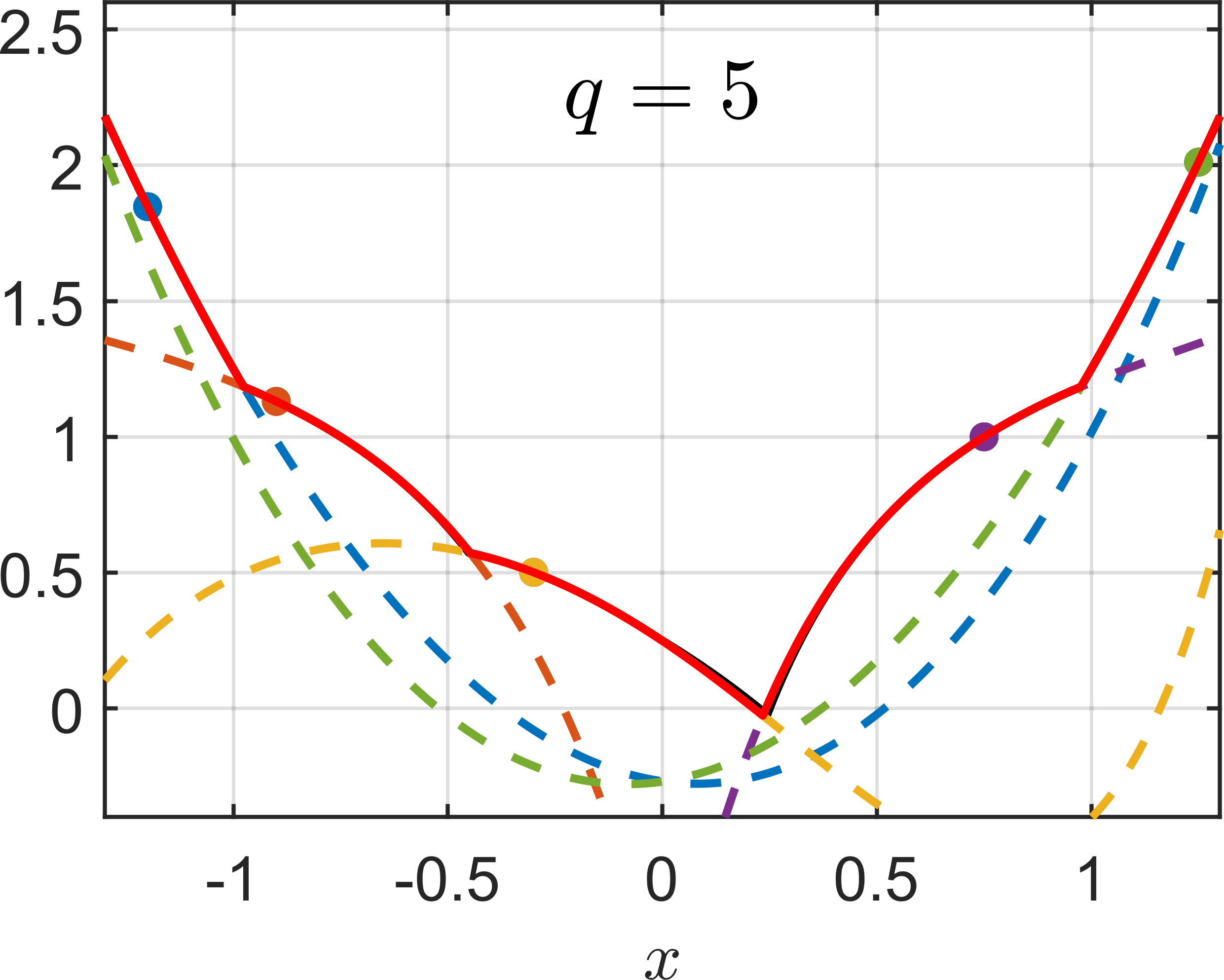}\\
    }
    \caption{Higher-order cutting-plane models (red) for the nonconvex function $f : \R \rightarrow \R$, $x \mapsto \max(\{ -(x + 0.5)^2 + 0.25 |x|^{3/2} + 0.5, x^2 + 0.5 |x|^{3/2} - 0.25, -1/(|x| + 0.25) + 2 \})$ for different orders $q$ of Taylor expansions (dashed) and the centers $W = \{ -1.2, -0.9, -0.3, 0.75, 1.25 \}$ (dots). The different colors for the Taylor expansions correspond to different centers.}
    \label{fig:sketch_model}
\end{figure}
Due to the max-type nature of these models, they only work well when $f$ itself has a max-type structure. As such, we restrict ourselves to the case where $f$ is a \lc{2} function, which means that it can locally be represented as the maximum of infinitely many $\C^2$ (or even $\C^\infty$) functions. (This class of functions is closely related to the class of \emph{weakly convex} and the class of \emph{prox-regular} functions, cf.\ \cite{DM2005}, Rem.\ 1.1.) In particular, the higher-order derivatives of these underlying smooth functions act as the ``higher-order generalized derivatives'' of $f$ at points where it is not differentiable. Since our models may be nonconvex in case $f$ is nonconvex, we only minimize them over a closed $\eps$-ball around the current iterate, so that the resulting method can be seen as a higher-order version of a trust-region bundle method (see, e.g., \cite{HUL1993b}, Sec.\ 2.1). In each iteration, it first generates a finite subset $W$ of the $\eps$-ball around the current iterate $x^j$ (via ``null steps'') for which the resulting model approximates $f$ sufficiently well, which can be measured using a Taylor-like error estimate. Afterwards, the minimum of this model is computed, yielding the next iterate $x^{j+1}$ (the ``serious step''). In terms of convergence, we prove that if $f$ satisfies a growth assumption, the initial trust-region radius is small enough, and the initial trust region contains the minimum $x^*$ of $f$, then the sequence $(x^j)_j$ converges R-superlinearly to the minimum. However, while this shows that our method is efficient in terms of oracle calls, we point out that the subproblem of minimizing the model is a non-quadratic, nonconvex (but smooth) optimization problem itself, which may be significantly more expensive to solve than the linear or quadratic subproblems in other methods.

It is important to note that the requirements of our local convergence result demand significantly more than just the initial $x^1$ being close to $x^*$, as we also have to explicitly know a small enough upper bound $\eps_1$ such that the initial trust region, i.e., the $\eps_1$-ball around $x^1$, contains $x^*$. As such, the question of how this initial data can be provided, or, in other words, how our method can be globalized, is crucial. To not overload the current work, we only give a sketch of the globalization here, and refer to \cite{GU2026b} (which was written in parallel) for the details. The idea is to use a global trust-region bundle method as a wrapper for the local method, by attempting the local method whenever the trust region in the global method is decreased. Using this idea, we can prove global convergence with transition to local R-superlinear convergence for certain finite max-type functions (cf.\ \cite{GU2026b}, Cor.\ 3.1).

The remainder of this work is structured as follows: In Sec.\ \ref{sec:preliminaries} we introduce the notation and the basic concepts that we use. Sec.\ \ref{sec:higher_order_cutting_plane_models} introduces the higher-order cutting-plane models and derives error estimates for the distance of the model minima to the actual minimum under a polynomial growth assumption. In Sec.\ \ref{sec:local_superlinear_method} we use these models to construct the trust-region bundle method and prove local R-superlinear convergence. The globalization of this method from \cite{GU2026b} is summarized in Sec.\ \ref{sec:globalization}. In Sec.\ \ref{sec:numerical_results} the R-superlinear convergence is verified in numerical experiments. Finally, in Sec.\ \ref{sec:discussion_and_outlook}, we discuss future work.

A Matlab implementation of our method (and the globalization from \cite{GU2026b}), including scripts for the reproduction of all experiments shown in this work, is available at \url{https://github.com/b-gebken/higher-order-trust-region-bundle-method}.

\section{Preliminaries} \label{sec:preliminaries}

    In this section, we introduce our notation and the class of objective functions that we consider. Let $\| \cdot \|$ be the Euclidean norm on $\R^n$. For $\eps \geq 0$ let $\Bcl_\eps(x) := \{ y \in \R^n : \| y - x \| \leq \eps \}$. The closure and the convex hull of a set $A \subseteq \R^n$ are denoted by $\cl(A)$ and $\conv(A)$, respectively. The sum of two sets $A, B \subseteq \R^n$ is defined as $A + B := \{ a + b : a \in A, b \in B \}$. 
    
    Since there is no standard notation for higher-order derivatives, we introduce the notation that we use here (from \cite{B1964}, Chapter V) for completeness. To this end, let $U \subseteq \R^n$ be open and $f : U \rightarrow \R$. For $q \in \N \cup \{ \infty \}$ we say that $f$ is $\C^q$ if it is $q$-times continuously differentiable at every $x \in U$. For $m \in \{ 1, \dots, q \}$ and $y, z, v \in \R^n$, denote $\deriv^{0} f(y)(v)^0 := f(y)$ and
    \begin{equation} \label{eq:def_taylor_series}
        \begin{aligned}
            \deriv^{m} f(y)(v)^m &:= \sum_{i_1 = 1}^n \cdots \sum_{i_m = 1}^n \partial_{i_1} \cdots \partial_{i_m} f(y) v_{i_1} \cdots v_{i_m}, \\
            T^q f(z,y) &:= \sum_{m = 0}^q \frac{1}{m!} \deriv^{m} f(y)(z - y)^m.
        \end{aligned}
    \end{equation}
    Note that $\deriv^{1} f(y)(v)^1 = \nabla f(y)^\top v$ and $\deriv^{2} f(y)(v)^2 = v^\top \nabla^2 f(y) v$. If $U$ is convex and $f$ is $\C^{q+1}$ for $q \in \N$, then by Taylor's theorem (see, e.g., \cite{B1964}, Thm.\ 20.16), for any $y, z \in U$ there is some $a \in \conv(\{ y, z \})$ such that $f(z) = T^q f(z,y) + R^{q+1}(z,y)$ for
    \begin{align} \label{eq:Taylor_remainder}
        R^{q+1}(z,y) = \frac{1}{(q+1)!} \deriv^{q+1} f(a)(z - y)^{q+1}.
    \end{align}
    An important observation for our results will be that continuity of the partial derivatives of $f$ up to order $q+1$ implies that for any bounded, convex set $V \subseteq U$, there is some $K > 0$ such that 
    \begin{align*}
        | R^{q+1}(z,y) | \leq K \| z - y \|^{q+1} \quad \forall y, z \in V.
    \end{align*}
    
    The objective functions we consider in this work belong to the class of \lc{q} functions \cite{RW1998}, which can be defined as follows:
    \begin{definition} \label{def:lower_Ck}
        A function $f : U \rightarrow \R$ is called \emph{lower-}$\C^q$ for $q \in \N \cup \{ \infty \}$, if, for every $\Bar{x} \in \R^n$, there are an open neighborhood $V \subseteq \R^n$ of $\Bar{x}$, a compact topological space $S$, and $\C^{q}$ functions $f_s : V \rightarrow \R$, $s \in S$, such that 
        \begin{align} \label{eq:lower_Ck_representation}
            f(x) = \max_{s \in S} f_s(x) \quad \forall x \in V
        \end{align}
        and $f_s$ and all its partial derivatives up to order $q$ depend continuously on $(s,x) \in S \times V$. 
    \end{definition}
    
    We refer to \eqref{eq:lower_Ck_representation} as a \emph{representation} of $f$ around $x$, with an \emph{index set} $S$. We call $A(x) := \{ s \in S : f(x) = f_s(x) \}$ the \emph{active set} of $f$ at $x$. The functions $f_s$, $s \in S$, are called the \emph{selection functions}, and we say that a selection function $f_s$ is \emph{active} at $x$ if $s \in A(x)$. If $S$ is finite, then we say that $f$ is a \emph{finite max-type function} on $V$. By \cite{RW1998}, Thm.\ 10.31, \lc{1} functions are locally Lipschitz continuous and their Clarke subdifferential \cite{C1990} is given by $\partial f(x) = \conv(\{ \nabla f_s(x) : s \in A(x) \})$. By \cite{RW1998}, Cor.\ 10.34, every \lc{2} function is automatically \lc{\infty}. For a general locally Lipschitz continuous functions $f : U \rightarrow \R$ and a point $x \in \R^n$, we say that $x$ is \emph{critical} if $0 \in \partial f(x)$. For a \lc{q} function, this means that there is a convex combination of gradients of active selection functions at $x$ that is zero.

\section{Higher-order cutting-plane models and error estimates} \label{sec:higher_order_cutting_plane_models}

    In this section, we first introduce higher-order cutting-plane models for the local approximation of \lc{q} functions. Afterwards, we derive an upper bound for the pointwise distance of these models to the original function. Combined with a growth assumption, this allows us to derive an estimate for the distance of the model minima to the actual minima. Finally, we discuss how these results can be used to construct solution methods with local R-superlinear convergence. 

    Before introducing the models, we first have to discuss the oracle information that we assume to be available. For first-order information, the standard oracle assumption for a locally Lipschitz continuous function $f$ is that for each $x$, we have access to $f(x)$ and a subgradient $\xi \in \partial f(x)$ (see, e.g., \cite{L1989}, (1.10)). If $f$ is \lc{q}, then any representation around $x$ yields the same first-order information $\partial f(x) = \conv(\{ \nabla f_s(x) : s \in A(x) \})$, so the subgradients are the convex combinations of gradients of active selection functions in any representation. In particular, if $x$ is a point where $f$ is $\C^1$, then $\partial f(x) = \{ \nabla f(x) \}$, so for any representation around $x$, all gradients of selection functions that are active at $x$ must equal $\nabla f(x)$. Unfortunately, for higher-order information, this relationship does not persist: for example, consider the function $f : (-1,1) \rightarrow \R$, $x \mapsto x^2$, which has the representation $f(x) = \max_{s \in S} f_s(x)$ with $S = [-1,1]$ and $f_s(x) = s^2 + 2s(x-s)$. Then for any $x \in (-1,1)$ we have $\nabla^2 f(x) = 2$,  but $\nabla^2 f_s(x) = 0$ for all $s \in A(x)$. More generally, by \cite{RW1998}, Thm.\ 10.33, for each \lc{2} function, there are representations with quadratic selection functions, such that no derivative information of $f$ of order $3$ or higher can be obtained from such selection functions. This means that the higher-order derivative information we obtain from selection functions heavily depends on the representation of $f$, and that it may not even correspond to higher-order derivatives of $f$ at smooth points.
    
    For the above reasons, we do not just assume $f : U \rightarrow \R$ to be a \lc{q} function on an open set $U \subseteq \R^n$, but also fix a single, global representation on $U$. By \cite{RW1998}, Prop.\ 10.54, if $U$ is bounded and $f$ can be extended to a \lc{q} function on an open superset $U'$ of the closure $\cl(U)$, then such a global representation always exists. In particular, if $f$ is a \lc{q} function on $\R^n$, then there is a global representation on any bounded set $U \subseteq \R^n$. Since the method we derive in this work always generates bounded sequences (cf.\ Lem.\ \ref{lem:global_convergence_local_algo}), assuming a large enough $U$ avoids any practical restrictions of this assumption. (Later on, in the local convergence results, $U$ can be thought of as a small open neighborhood of the minimum of $f$.) Furthermore, to be able to use the remainder formula \eqref{eq:Taylor_remainder}, we assume that $U$ is convex and that the selection functions are $\C^{q+1}$. More formally, for $f : U \rightarrow \R$, consider the following assumption:
    \begin{assum} \label{assum:A1}
        The set $U$ is open and convex. For $q \in \N$ there are a compact topological space $S$ and $\C^{q+1}$ functions $f_s : U \rightarrow \R$, $s \in S$, such that 
        \begin{align*}
            f(x) = \max_{s \in S} f_s(x) \quad \forall x \in U
        \end{align*}
        and $f_s$ and all its partial derivatives up to order $q+1$ depend continuously on $(s,x) \in S \times U$. 
    \end{assum}

    Clearly, \ref{assum:A1} implies that $f$ is \lc{q+1} on $U$, and, since $q+1 \geq 2$, it is even \lc{\infty}. We assume that we have access to the following oracle information for functions satisfying \ref{assum:A1}:
    \begin{oracle} \label{oracle:1}
        For a function $f : U \rightarrow \R$ satisfying \ref{assum:A1} and for each $x \in U$, we have access to the objective value $f(x)$ and the maps $v \mapsto \deriv^{m} f_{s(x)}(x)(v)^m$ for some $s(x) \in A(x)$ and all $m \in \{1, \dots, q\}$.
    \end{oracle}
    
    We emphasize that the oracle only implies that we have access to the derivatives of $f_{s(x)}$ at $x$ up to order $q$, but not to the index $s(x)$ itself or other information about the function $f_{s(x)}$. In particular, when evaluating $f$ or its derivatives in two different points $x^1, x^2 \in U$, we do not know whether $s(x^1) = s(x^2)$. (In a numerical setting, one typically does not encounter points at which $f$ is not $\C^\infty$ unless specific initial data is chosen. As such, for the numerical experiments in Sec.\ \ref{sec:numerical_results}, we simply use the derivatives of $f$ itself. The discrepancy of this ``practical oracle'' to Oracle \ref{oracle:1} is discussed in Sec.\ \ref{sec:discussion_and_outlook}.)

    Using the information provided by Oracle \ref{oracle:1}, the $q$\emph{-order cutting-plane model} can be defined as follows: Let $q \in \N$ and assume that $f : U \rightarrow \R$ satisfies \ref{assum:A1}. For $x \in U$, $\eps \geq 0$ with $\Bcl_\eps(x) \subseteq U$, a nonempty, finite set $W \subseteq \Bcl_\eps(x)$, and $z \in \R^n$, let
    \begin{align} \label{eq:def_model}
        \T^{q,W}(z) 
        := \max_{y \in W} T^q f_{s(y)}(z,y)
        = \max_{y \in W} \sum_{m = 0}^q \frac{1}{m!} \deriv^{m} f_{s(y)}(y)(z - y)^m.
    \end{align}
    Since $W$ is finite and $z \mapsto T^q f_{s(y)}(z,y)$ is $\C^\infty$ for all $y \in W$, the function $\T^{q,W}$ is \lc{\infty} and, in particular, locally Lipschitz continuous. Fig.\ \ref{fig:sketch_model} shows the graph of $\T^{q,W}$ (red) for $q \in \{1,2,5\}$.
    
    In the following, we derive an error estimate for these models. To this end,
    denote $s(W) := \{ s(y) : y \in W \} \subseteq S$ and
    \begin{align*}
        f^W(z) := \max_{y \in W} f_{s(y)}(z), \quad \calR^{q,W}(z) := f^W(z) - \T^{q,W}(z).
    \end{align*}
    Clearly, $f^W = f$ if $s(W) = S$. In a sense, $f^W$ is the best approximation of $f$ we can hope to achieve when only using the oracle information at points from $W$. By applying Taylor's theorem to each selection function, we obtain the following upper bound for the error $|\calR^{q,W}(z)|$ of the model $\T^{q,W}$ in $\Bcl_\eps(x)$:
    
    \begin{lemma} \label{lem:q_order_error_estimate}
        Let $q \in \N$ and assume that $f : U \rightarrow \R$ satisfies \ref{assum:A1}. Then for every bounded set $V \subseteq U$ and every $\epsmax > 0$ with $\cl(V + \Bcl_{\epsmax}(0)) \subseteq U$, there is some $K \geq 0$ such that
        \begin{align*}
            \max_{z \in \Bcl_\eps(x)} |\calR^{q,W}(z)| \leq K \eps^{q+1}
        \end{align*}
        for all $x \in V$, $\eps \in [0,\epsmax]$, and finite, nonempty sets $W \subseteq \Bcl_\eps(x)$.
    \end{lemma}
    \begin{proof}
        Assume w.l.o.g.\ that $V$ is convex. (Recall that $U$ is convex by assumption.) \\
        \textbf{Part 1:} Let $s \in S$ and $y \in V + \Bcl_{\epsmax}(0)$. Taylor's theorem (cf.\ Sec.\ \ref{sec:preliminaries}) applied to $f_s$ shows that for any $z \in V + \Bcl_{\epsmax}(0)$, there is some $a \in \conv(\{ y,z \}) \subseteq V + \Bcl_{\epsmax}(0)$ such that
        \begin{align*}
            f_s(z) - T^q f_s(z,y) = \frac{1}{(q+1)!} \deriv^{q+1} f_s(a)(z - y)^{q+1}.
        \end{align*}
        Continuity of partial derivatives of $f$ up to order $q+1$ with respect to $(s,x)$, compactness of $S$, and compactness of $\cl(V + \Bcl_{\epsmax}(0))$ imply that there is an upper bound for the derivative on the right-hand side of this inequality that does not depend on $s$ or $a$ (but on $V$, $\epsmax$, and $q$). More formally, there is some $K' \geq 0$ such that
        \begin{align} \label{eq:proof_lem_q_order_error_estimate_1}
            | f_s(z) - T^q f_s(z,y) | \leq \frac{K'}{(q+1)!} \| z - y \|^{q+1}
            \quad \forall y, z \in V + \Bcl_{\epsmax}(0), s \in S.
        \end{align} 
        \textbf{Part 2:} Let $x \in V$ and $\eps \in [0,\epsmax]$. Since $\| z - y \|^{q+1} \leq 2^{q+1} \eps^{q+1}$ for all $z, y \in \Bcl_\eps(x)$, \eqref{eq:proof_lem_q_order_error_estimate_1} shows that
        \begin{align} \label{eq:proof_lem_q_order_error_estimate_2}
            |f_s(y) - T^q f_s(z,y)| \leq K \eps^{q+1} \quad \forall z, y \in \Bcl_\eps(x), s \in S
        \end{align}
        for $K := (2^{q+1} K')/(q+1)!$. \\
        \textbf{Part 3:} Let $x \in V$, $\eps \in [0,\epsmax]$, and $z \in \Bcl_\eps(x)$. Let $y^1 \in W$ be such that $f^W(z) = f_{s(y^1)}(z)$ and $y^2 \in W$ be such that $\T^{q,W}(z) = T^q f_{s(y^2)}(z,y^2)$. If $\calR^{q,W}(z) < 0$, then \eqref{eq:proof_lem_q_order_error_estimate_2} and $f^W(z) \geq f_{s(y^2)}(z)$ imply that
        \begin{align*}
            |\calR^{q,W}(z)|
            &= -\calR^{q,W}(z)
            = \T^{q,W}(z) - f^W(z)
            = T^q f_{s(y^2)}(z,y^2) - f^W(z) \\
            &\leq T^q f_{s(y^2)}(z,y^2) - f_{s(y^2)}(z)
            \leq K \eps^{q+1}.
        \end{align*}
        If instead $\calR^{q,W}(z) \geq 0$, then \eqref{eq:proof_lem_q_order_error_estimate_2} and $\T^{q,W}(z) \geq T^q f_{s(y^1)}(z,y^1)$ imply that
        \begin{align*}
            |\calR^{q,W}(z)| 
            &= \calR^{q,W}(z) 
            = f^W(z) - \T^{q,W}(z)
            = f_{s(y^1)}(z) - \T^{q,W}(z) \\
            &\leq f_{s(y^1)}(z) - T^q f_{s(y^1)}(z,y^1)
            \leq K \eps^{q+1},
        \end{align*}
        completing the proof.
    \end{proof}

    The idea of our minimization algorithm is to approximate the minimum of $f$ by minimizing $\T^{q,W}$.
    Since $\T^{q,W}$ may be nonconvex and since the error estimate in Lem.\ \ref{lem:q_order_error_estimate} only holds on $\Bcl_\eps(x)$, we  constrain the minimization of $\T^{q,W}$ to $\Bcl_\eps(x)$. As such, our approach can be seen as a type of trust-region method. More formally, let $q \in \N$ and assume that $f : U \rightarrow \R$ satisfies \ref{assum:A1}. For $x \in U$, $\eps \geq 0$ with $\Bcl_\eps(x) \subseteq U$, and a nonempty, finite set $W \subseteq \Bcl_\eps(x)$, let
    \begin{align} \label{eq:def_bar_z}
        \bar{z}^{q,W}(x,\eps) \in \argmin_{z \in \Bcl_\eps(x)} \T^{q,W}(z), 
    \end{align}
    which is well-defined by continuity of $\T^{q,W}$. For the sake of brevity, we write $\bar{z}^W = \bar{z}^{q,W}(x,\eps)$ whenever the context allows. While the optimization problem on the right-hand side of \eqref{eq:def_bar_z} is again a nonsmooth problem, it has the following epigraph reformulation as a smooth, constrained problem:
    \begin{equation} \label{eq:bar_z_epigraph}
        \begin{aligned}
            \min_{z \in \R^n, \theta \in \R} \ & \theta \\
            \text{s.t.} \ & T^q f_{s(y)}(z,y) \leq \theta \quad \forall y \in W,\\
            & \| z - x \|^2 \leq \eps^2.
        \end{aligned}
    \end{equation}
    Note that for $q \geq 2$ this problem is neither quadratic nor convex. As such, it may be significantly more expensive to solve than the subproblems that appear in common bundle methods.

    Assume that $f$ has a minimum $x^*$ in $U$. In the following, we derive an upper bound for $\| \bar{z}^{q,W}(x,\eps) - x^* \|$ when $x \in \Bcl_\eps(x^*)$, i.e., when $x^*$ lies inside the trust region $\Bcl_\eps(x)$. By Lem.\ \ref{lem:q_order_error_estimate}, in the ideal case where $s(W) = S$, the model $\T^{q,W}$ approximates $f$ on $\Bcl_\eps(x)$ up to an error of $K \eps^{q+1}$. As such, if $f(x) - f(x^*) \leq K \eps^{q+1}$ for $x \in \Bcl_\eps(x^*)$, then the error bound in Lem.\ \ref{lem:q_order_error_estimate} cannot be used to show that the point $\bar{z}^{q,W}(x,\eps)$ is in any way more favorable than the original point $x$. To circumvent this issue, we have to make sure that $f$ does not become too ``flat'' around its minimum, which we do via the following growth assumption:
    \begin{assum} \label{assum:A2}
        The function $f : U \rightarrow \R$ satisfies \ref{assum:A1} for $q \in \N$.
        Furthermore, for $p \in \N$ and $x^* \in U$, there is some $\beta > 0$ such that
        \begin{align*}
            f(x) \geq f(x^*) + \beta \| x - x^* \|^p
        \end{align*}
        for all $x \in U$. The value $p$ is referred to as the \emph{order of growth} of $f$ around $x^*$.
    \end{assum}

    Note that \ref{assum:A2} implies that $x^*$ is the unique global minimum of $f$ in $U$, and that for bounded $U$, a minimum of order $p$ is also a minimum of any order $p' \geq p$. (Typically, growth assumptions only have to hold on an open neighborhood of a point. However, since $U$ can be thought of a small open neighborhood of the minimum in all our local convergence results, the global growth on $U$ in \ref{assum:A2} is no practical restriction.) If $s(W) = S$ and $q \geq p$, then \ref{assum:A2} avoids the issues discussed above. For $s(W) \neq S$, an additional bound for $f(\bar{z}^W) - \T^{q,W}(\bar{z}^W)$ has to be assumed. In general, we obtain the following lemma:
    \begin{lemma} \label{lem:T_W_error_estimate}
        Assume that $f : U \rightarrow \R$ satisfies \ref{assum:A2} for $q, p \in \N$. Denote $\bar{z}^W = \bar{z}^{q,W}(x,\eps)$. Then for every $\epsmax > 0$ with $\Bcl_{2\epsmax}(x^*) \subseteq U$ there is some $K \geq 0$ such that for every $\eps \in [0,\epsmax]$, $x \in \Bcl_\eps(x^*)$, and finite, nonempty set $W \subseteq \Bcl_\eps(x)$, it holds
        \begin{equation} \label{eq:T_W_error_estimate_1}
            \begin{aligned}
                \beta \| \bar{z}^W - x^* \|^p
                &\leq f(\bar{z}^W) - \T^{q,W}(\bar{z}^W) + K \eps^{q+1} \\
                &\leq f(\bar{z}^W) - f^W(\bar{z}^W) + 2K \eps^{q+1},
            \end{aligned}
        \end{equation}
        In particular:
        \begin{enumerate}[label=(\alph*)]
            \item If $A(\bar{z}^W) \cap s(W) \neq \emptyset$ then
            \begin{align*}
                \| \bar{z}^W - x^* \| \leq \left( \frac{2K}{\beta} \right)^{1/p} \eps^{\frac{q+1}{p}}.
            \end{align*}
            \item If $\sigma \in (0,1]$ and
            \begin{align} \label{eq:T_W_error_estimate_2}
                f(\bar{z}^W) - \T^{q,W}(\bar{z}^W) \leq \eps^{q+\sigma},
            \end{align}
            then
            \begin{align} \label{eq:T_W_error_estimate_3}
                \| \bar{z}^W - x^* \| \leq \left( \frac{1 + K \eps^{1 - \sigma}}{\beta} \right)^{1/p} \eps^{\frac{q+\sigma}{p}}.
            \end{align}
        \end{enumerate}
    \end{lemma}
    \begin{proof}
        Let $\epsmax > 0$ so that $\Bcl_{2 \epsmax}(x^*) \subseteq U$. Let $\eps \in [0,\epsmax]$, $x \in \Bcl_\eps(x^*)$, and $W \subseteq \Bcl_\eps(x)$ be finite and nonempty. Then $\Bcl_\eps(x) \subseteq \Bcl_{2 \epsmax}(x^*) \subseteq U$ and
        \begin{align} \label{eq:proof_lem_T_W_error_estimate}
            f(x^*) 
            \geq f^W(x^*) 
            = \T^{q,W}(x^*) + \calR^{q,W}(x^*) 
            \geq \T^{q,W}(\bar{z}^W) + \calR^{q,W}(x^*).
        \end{align}
        Since $\bar{z}^W \in \Bcl_\eps(x) \subseteq U$, the growth assumption \ref{assum:A2} yields
        \begin{align*}
            \beta \| \bar{z}^W - x^* \|^p
            \leq f(\bar{z}^W) - f(x^*)
            \leq f(\bar{z}^W) - \T^{q,W}(\bar{z}^W) - \calR^{q,W}(x^*).
        \end{align*}
        Applying Lem.\ \ref{lem:q_order_error_estimate} to $f^W$ (for $V = \Bcl_{\epsmax}(x^*)$) yields some $K \geq 0$ (which does not depend on $x$, $\eps$, or $W$) such that
        \begin{align*}
            \beta \| \bar{z}^W - x^* \|^p
            \leq f(\bar{z}^W) - \T^{q,W}(\bar{z}^W) + K \eps^{q+1},
        \end{align*}
        which shows the first inequality in \eqref{eq:T_W_error_estimate_1}. The second inequality follows by
        \begin{align*}
            f(\bar{z}^W) - \T^{q,W}(\bar{z}^W) + K \eps^{q+1}
            &= f(\bar{z}^W) - f^W(\bar{z}^W) + \calR^{q,W}(\bar{z}^W) + K \eps^{q+1} \\
            &\leq f(\bar{z}^W) - f^W(\bar{z}^W) + 2K \eps^{q+1}.
        \end{align*}
        If $A(\bar{z}^W) \cap s(W) \neq \emptyset$ then $f(\bar{z}^W) - f^W(\bar{z}^W) = 0$, so $\beta \| \bar{z}^W - x^* \|^p \leq 2K \eps^{q+1}$, which is equivalent to the estimate in (a). If $f(\bar{z}^W) - \T^{q,W}(\bar{z}^W) \leq \eps^{q+\sigma}$ then
        \begin{align*}
            \beta \| \bar{z}^W - x^* \|^p \leq \eps^{q + \sigma} + K \eps^{q+1} = (1 + K \eps^{1 - \sigma}) \eps^{q + \sigma},
        \end{align*}
        which is equivalent to the estimate in (b), completing the proof.
    \end{proof}

    Lem.\ \ref{lem:T_W_error_estimate} shows that if we know that the distance of $x$ to the minimum $x^*$ is at most $\eps$ and one of the prerequisites of (a) or (b) holds, then the distance $\| \bar{z}^W - x^* \|$ is at most $M \eps^{(q+\sigma)/p}$ for $\sigma \in (0,1]$ and some $M > 0$ which does not depend on $x$, $\eps$, or $W$ (since $\eps \leq \epsmax$ in the first factor on the right-hand side of \eqref{eq:T_W_error_estimate_3}). In other words, $\| \bar{z}^W - x^* \|$ and the estimate $\eps$ for $\| x - x^* \|$ differ by a factor of $M \eps^{((q+\sigma)/p) - 1}$. If $q \geq p$ and $\eps$ is small enough, then this factor is less than $1$, such that the distance $\| \bar{z}^W - x^* \|$ is less than the estimate for the distance $\| x - x^* \|$. Starting with some $x^1 \in U$ and $\eps_1 > 0$ such that $x^* \in \Bcl_{\eps_1}(x^1)$, this motivates a method for approximating $x^*$ by iterating $x^{j+1} = \bar{z}^{q,W_j}(x^j,\eps_j)$ for suitable sequences $(W_j)_j$ and $(\eps_j)_j$ with $\eps_j \rightarrow 0$ and $\| \bar{z}^{q,W_j}(x^j,\eps_j) - x^* \| \leq \eps_{j+1}$ for all $j \in \N$. Since the factor $M \eps^{((q+\sigma)/p) - 1}$ decreases when $\eps$ decreases, $(\eps_j)_j$ can be chosen as Q-superlinearly vanishing, such that $(x^j)_j$ converges R-superlinearly to $x^*$. To obtain an implementable method from this idea, there are three challenges that have to be overcome:
    \begin{enumerate}[leftmargin=1.1cm,label=(C\arabic*)]
        \item \label{enum:C1} Every iteration requires a compact set $W_j \subseteq \Bcl_{\eps_j}(x^j)$ such that the prerequisites of (a) or (b) in Lem.\ \ref{lem:T_W_error_estimate} are satisfied. The prerequisite of (a) is satisfied trivially if $S$ is finite and $s(W_j) = S$, i.e., if for each $s' \in S$, $W_j$ contains a point $y$ with $s(y) = s'$. However, recall that our oracle does not give us access to any information about the active indices of $f$, which makes it impossible to work with prerequisite of (a) explicitly. Fortunately, no knowledge about active indices is required when working with the prerequisite of (b), and the left-hand side of \eqref{eq:T_W_error_estimate_2} can be evaluated in practice.
        \item \label{enum:C2} A vanishing sequence $(\eps_j)_j$ with Q-superlinear convergence has to be found such that $\| \bar{z}^{q,W_j}(x^j,\eps_j) - x^* \| = M \eps_j^{(q + \sigma)/p} \leq \eps_{j+1}$ for all $j \in \N$. In theory, we could simply define $\eps_{j+1}$ as $M \eps_j^{(q + \sigma)/p}$. However, since the constant $M$ depends on $K$ (from Lem.\ \ref{lem:q_order_error_estimate}) and $\beta$ (from \ref{assum:A2}), and since we do not assume that these two constants are known, we cannot do this in practice. Instead, $\eps_{j+1}$ has to be an upper bound that is tight enough for $(\eps_j)_j$ to be Q-superlinearly vanishing.
        \item \label{enum:C3} The method requires an initial point $x^1$ that is already close enough to the minimum $x^*$. Additionally, a sufficiently small $\eps_1$ with $x^* \in \Bcl_{\eps_1}(x^1)$ has to be known.
    \end{enumerate}
    We will refer to the items in the above list as Challenges \ref{enum:C1}, \ref{enum:C2}, and \ref{enum:C3}. In the following section, we show how \ref{enum:C1} and \ref{enum:C2} can be overcome to obtain an implementable, locally convergent method. Challenge \ref{enum:C3} is concerned with the globalization of the local method. We only give a summary of how this can be achieved in Sec.\ \ref{sec:globalization} and refer to the accompanying paper \cite{GU2026b} for the details.

\section{Trust-region bundle method with R-superlinear convergence} \label{sec:local_superlinear_method}

    In this section, we turn the theoretical method described at the end of the previous section into a practical method with local R-superlinear convergence by resolving the Challenges \ref{enum:C1} and \ref{enum:C2}. First of all, for \ref{enum:C1}, we present a subroutine that computes a set $W$ satisfying the prerequisite \eqref{eq:T_W_error_estimate_2} of Lem.\ \ref{lem:T_W_error_estimate}(b) by iteratively solving \eqref{eq:bar_z_epigraph}. Afterwards, for \ref{enum:C2}, we construct an explicit sequence $(\eps_j)_j$ that has the required properties if the initial $\eps_1$ is small enough.

    To this end, let $x \in U$, $\eps \geq 0$ with $\Bcl_\eps(x) \subseteq U$, and let $W \subseteq \Bcl_{\eps}(x)$ be finite and nonempty. By definition of $\T^{q,W}$ (cf.\ \eqref{eq:def_model}), for all $y \in W$ we have
    \begin{align*}
        \T^{q,W}(y) \geq T^q f_{s(y)}(y,y) = f(y),
    \end{align*}
    so 
    \begin{align*}
        f(y) - \T^{q,W}(y) \leq 0 \leq \eps^{q+\sigma} \quad \forall y \in W.
    \end{align*}
    In particular, if \eqref{eq:T_W_error_estimate_2} is violated, then $\bar{z}^W \notin W$. Thus, adding $\bar{z}^W$ to $W$ leads to an augmented model. This is the motivation for \algApproxW{}.
    \begin{algorithm} 
        \caption{Compute $W$ satisfying \eqref{eq:T_W_error_estimate_2}}
        \label{algo:approx_W}
        \begin{algorithmic}[1] 
            \Require Oracle \ref{oracle:1}, point $x \in U$, radius $\eps > 0$ with $\Bcl_\eps(x) \subseteq U$, $q \in \N$, finite, nonempty set $W^1 \subseteq \Bcl_\eps(x)$, $\sigma \in (0,1)$.
            \For{$i = 1, 2, \dots$}
                \State Compute $\bar{z}^{W^i} = \bar{z}^{q,W^i}(x,\eps)$ (cf.\ \eqref{eq:def_bar_z}, \eqref{eq:bar_z_epigraph}). \label{state:approx_W_solve_subproblem}
                \If{$f(\bar{z}^{W^i}) - \T^{q,W^i}(\bar{z}^{W^i}) \leq \eps^{q + \sigma}$} \label{state:approx_W_stopping_criterion}
                    \State Stop.
                \Else
                    \State Set $W^{i+1} = W^i \cup \{ \bar{z}^{W^i} \}$.
                \EndIf
            \EndFor
        \end{algorithmic}
    \end{algorithm}	
    The following lemma shows that this algorithm always terminates, which means that a set satisfying \eqref{eq:T_W_error_estimate_2} for a given $\sigma \in (0,1)$ is found:
    
    \begin{lemma} \label{lem:algo_approx_W_termination}
        Let $q \in \N$ and assume that $f : U \rightarrow \R$ satisfies \ref{assum:A1}. Let $x \in U$.
        \begin{enumerate}[label=(\alph*)]
            \item Let $\eps > 0$ with $\Bcl_\eps(x) \subseteq U$. Then \algApproxW{} terminates.
            \item If $S$ is finite then there is some $\epsmax > 0$ such that for all $\eps \in (0,\epsmax]$, it holds $\Bcl_\eps(x) \subseteq U$ and \algApproxW{} terminates in at most $|S|$ iterations.
        \end{enumerate}
    \end{lemma}
    \begin{proof}
        \textbf{(a)} Assume that \algApproxW{} does not terminate, i.e., that the inequality in Step \ref{state:approx_W_stopping_criterion} is violated for all $i \in \N$. Then $(\bar{z}^{W^i})_i \subseteq \Bcl_\eps(x)$ is an infinite sequence, which, by compactness of $\Bcl_\eps(x)$, has an accumulation point.
        Let $(\hat{z}^l)_l$ be a converging subsequence of $(\bar{z}^{W^i})_i$ and let $(\hat{W}^l)_l$ be the corresponding subsequence of $(W^i)_i$. Then by definition of $\T^{q,W}$ (cf.\ \eqref{eq:def_model}) and since $\hat{z}^{l-1} \in \hat{W}^l$, it holds
        \begin{equation*}
            \begin{aligned}
                f(\hat{z}^l) - \T^{q,\hat{W}^l}(\hat{z}^l)
                &\leq f(\hat{z}^l) - T^q f_{s(\hat{z}^{l-1})}(\hat{z}^l,\hat{z}^{l-1}) \\
                &= f(\hat{z}^l) - \sum_{m = 0}^q \frac{1}{m!} \deriv^{m} f_{s(\hat{z}^{l-1})}(\hat{z}^{l-1})(\hat{z}^l - \hat{z}^{l-1})^m \\
                &= f(\hat{z}^l) - f(\hat{z}^{l-1})  + \sum_{m = 1}^q \frac{1}{m!} \deriv^{m} f_{s(\hat{z}^{l-1})}(\hat{z}^{l-1})(\hat{z}^l - \hat{z}^{l-1})^m.
            \end{aligned}
        \end{equation*}
        By continuity of $f$ and $(s,x) \mapsto \deriv^{m} f_s(x)$, and by compactness of $S$, the right-hand of this inequality vanishes for $l \rightarrow \infty$. In particular, the inequality in Step \ref{state:approx_W_stopping_criterion} has to hold after finitely many iterations, which is a contradiction. \\
        \textbf{(b)} If $f(\bar{z}^{W^i}) - f^{W^i}(\bar{z}^{W^i}) > 0$ holds in iteration $i$ of \algApproxW{}, then $s(\bar{z}^{W^i}) \notin s(W^i)$. By construction of the algorithm, this means that $|s(W^{i+1})| = |s(W^i)| + 1$. Since $s(W^i) \subseteq S$ and $|s(W^1)| = |W^1| \geq 1$, this can only happen in at most $|S| - 1$ iterations. In particular, there must be some $i \in \{ 1, \dots, |S| \}$ with $f(\bar{z}^{W^i}) - f^{W^i}(\bar{z}^{W^i}) = 0$.
        Let $\epsmax > 0$ so that $\Bcl_{\epsmax}(x) \subseteq U$. Let $\eps \in (0,\epsmax]$. Then $\Bcl_\eps(x) \subseteq U$ and Lem.\ \ref{lem:q_order_error_estimate} (with $V = \Bcl_{\epsmax}(x)$) implies that there is some $K \geq 0$ (which does not depend on $\eps$ or $W^i$) such that
        \begin{align*}
            0
            &= f(\bar{z}^{W^i}) - f^{W^i}(\bar{z}^{W^i})
            = f(\bar{z}^{W^i}) - \T^{q,W^i}(\bar{z}^{W^i}) - \calR^{q,W^i}(\bar{z}^{W^i}) \\
            &\geq f(\bar{z}^{W^i}) - \T^{q,W^i}(\bar{z}^{W^i}) - K \eps^{q+1},
        \end{align*}
        so
        \begin{align*}
            f(\bar{z}^{W^i}) - \T^{q,W^i}(\bar{z}^{W^i}) \leq K \eps^{q+1}.
        \end{align*}
        Assuming w.l.o.g.\ that $\epsmax < (1/K)^{1/(1-\sigma)}$ implies $K \eps^{q+1} < \eps^{q + \sigma}$, causing the algorithm to stop in Step \ref{state:approx_W_stopping_criterion} in iteration $i \leq |S|$. (If $K = 0$ then this follows trivially.)
    \end{proof}
    
    We discuss further properties of \algApproxW{} in the following remark:
    \begin{remark} \label{rem:bundle_initialization}
        \begin{enumerate}[label=(\alph*)]
            \item A simple choice for the initial $W^1$ is $W^1 = \{ x \}$. However, if \algApproxW{} is used as a subroutine in a larger algorithm, then points from $\Bcl_\eps(x)$ in which the oracle was already evaluated can be included in $W^1$. In this way, a bundle-like behavior with a memory of oracle information can be induced. Alternatively, one could randomly sample points from $\Bcl_\eps(x)$ for the initial $W^1$ as in random gradient sampling \cite{BLO2005,BCL2020}.
            \item Since the upper bound $|S|$ in Lem.\ \ref{lem:algo_approx_W_termination}(b) may be large, the hope is that in practice, \algApproxW{} terminates in far fewer iterations. However, unfortunately, numerical experiments will later suggest that this bound is sharp (cf.\ Ex.\ \ref{example:max_root}, where $|S| = 2n$).
        \end{enumerate}
    \end{remark}

    \algApproxW{} resolves Challenge \ref{enum:C1}, since it allows us to compute finite sets $W$ for which the estimate \eqref{eq:T_W_error_estimate_3} in Lem.\ \ref{lem:T_W_error_estimate}(b) holds. To resolve Challenge \ref{enum:C2}, we have to find an upper bound for this estimate that we can actually compute in practice and that is tight enough to obtain fast convergence. To this end, for $\eps_1 > 0$ and $\sigma, \kappa \in (0,1)$, consider the sequence $(\eps_j)_j$ defined by
    \begin{align} \label{eq:def_eps_j}
        \eps_j := \eps_1 \kappa^{(\frac{q+\sigma}{p})^{j-1} - 1} \quad \forall j \in \N.
    \end{align}
    The following, purely arithmetic lemma shows that if $q \geq p$, then for any choice of $\sigma$ and $\kappa$, this sequence has the desired properties if $\eps_1$ is small enough:

    \begin{lemma} \label{lem:eps_estimate}
        Let $q,p \in \N$ with $q \geq p$, $\eps_1 > 0$, and $\sigma, \kappa \in (0,1)$. Let $(\eps_j)_j$ be defined as in \eqref{eq:def_eps_j}.
        \begin{enumerate}[label=(\alph*)]
            \item The sequence $(\eps_j)_j$ monotonically decreases and vanishes Q-superlinearly with order $(q+\sigma)/p$.
            \item Let $M > 0$. Then there is some $\epsmax > 0$ such that for all $\eps_1 \in (0,\epsmax]$, it holds
            \begin{align*}
                M \eps_j^{\frac{q+\sigma}{p}} < \eps_{j+1}
                \quad \forall j \in \N
                \quad \text{and} \quad
                M \eps_{j-1}^{\frac{q+\sigma}{p}} + M \eps_j^{\frac{q+\sigma}{p}} 
                < \eps_j \quad \forall j \geq 2.
            \end{align*}
        \end{enumerate}
    \end{lemma}
    \begin{proof}
        For ease of notation let $Q := \frac{q + \sigma}{p}$, so $\eps_j = \eps_1 \kappa^{Q^{j-1} - 1}$. Since $q \geq p$ and $\sigma \in (0,1)$ it holds $Q > 1$. \\
        \textbf{(a)} Let $j \in \N$ and $a > 0$. Then
        \begin{align*}
            \frac{\eps_{j+1}}{\eps_j^a}
            &= \frac{\eps_1 \kappa^{Q^j - 1}}{\eps_1^a \kappa^{a ( Q^{j-1} - 1)}}
            = \eps_1^{1 - a} \kappa^{Q^j - 1 - a ( Q^{j-1} - 1)} \\
            &= \eps_1^{1 - a} \kappa^{a-1} \kappa^{Q^j - a Q^{j-1}}
            = \eps_1^{1 - a} \kappa^{a-1} \kappa^{Q^{j-1} (Q - a)}.
        \end{align*}
        For $a = 1$, the right-hand side of this equation is less than $1$ for all $j \in \N$ and vanishes for $j \rightarrow \infty$, since $\kappa \in (0,1)$ and $Q > 1$. Thus, $(\eps_j)_j$ is monotonically decreasing and vanishes Q-superlinearly. Furthermore, for $a = Q = (q+\sigma)/p$, the right-hand side does not depend on $j$ and is therefore bounded, such that $(\eps_j)_j$ converges with order $(q+\sigma)/p$. \\
        \textbf{(b)} It holds
        \begin{align*}
            M \eps_j^{Q}
            &= M \eps_1^{Q} \kappa^{Q (Q^{j-1} - 1)}
            = M \eps_1^{Q} \kappa^{-Q} \kappa^{Q^j} \\
            &= M \eps_1^{Q - 1} \kappa^{-Q+1} \eps_1 \kappa^{Q^j - 1}
            = M \eps_1^{Q - 1} \kappa^{-Q+1} \eps_{j+1}
            \quad \forall j \in \N.
        \end{align*}
        Since $Q > 1$, we can choose $\epsmax$ small enough so that for all $\eps_1 \in (0,\epsmax]$, it holds
        \begin{align} \label{eq:proof_lem_eps_estimate_1}
            M \eps_j^{Q} < \frac{1}{2} \eps_{j+1} \quad \forall j \in \N.
        \end{align}
        In particular, the first inequality in \eqref{lem:eps_estimate} holds. For the second inequality, for $\eps_1 \in (0,\epsmax]$, we have
        \begin{align*}
            M \eps_{j-1}^{Q} + M \eps_j^{Q} 
            \stackrel{\eqref{eq:proof_lem_eps_estimate_1}}{<} \frac{1}{2} \eps_j + M \eps_j^{Q} 
            = \left( \frac{1}{2} + M \eps_j^{Q - 1} \right) \eps_j
            \quad \forall j \geq 2.
        \end{align*}
        Since $(\eps_j)_j$ is monotonically decreasing, this shows that we can choose $\epsmax$ small enough so that for all $\eps_1 \in (0,\epsmax]$, the second inequality in \eqref{lem:eps_estimate} holds.
    \end{proof}

    The previous lemma resolves Challenge \ref{enum:C2}, and the resulting method is \algLocal{}.
    \begin{algorithm} 
        \caption{Local superlinear method}
        \label{algo:local_method}
        \begin{algorithmic}[1] 
            \Require Oracle \ref{oracle:1}, initial point $x^1 \in U$, initial radius $\eps_1 > 0$, model order $q \in \N$, growth order $p \in \N$, parameters $\sigma, \kappa \in (0,1)$.
            \For{$j = 1, 2, \dots$}
                \State Set $\eps_j = \eps_1 \kappa^{(\frac{q + \sigma}{p})^{j-1} - 1}$ (cf.\ \eqref{eq:def_eps_j}).
                \State Compute $W_j \subseteq \Bcl_{\eps_j}(x^j)$ via \algApproxW{} (with initialization $W^1 = \{ x^j \}$). \label{state:local_method_approx_W}
                \State Set $x^{j+1} = \bar{z}^{q,W_j}(x^j,\eps_j)$ (cf.\ \eqref{eq:def_bar_z}, \eqref{eq:bar_z_epigraph}). \label{state:local_method_new_iterate}
            \EndFor
        \end{algorithmic}
    \end{algorithm}
    By construction of \algApproxW{}, $\bar{z}^{q,W_j}(x^j,\eps_j)$ in Step \ref{state:local_method_new_iterate} was already computed in Step \ref{state:local_method_approx_W}. While knowledge about the order of growth $p$ is required, it suffices if $p$ is an upper estimate for the actual order when considering the local convergence (cf.\ \ref{assum:A2}). Note that there is no mechanism in \algLocal{} that enforces that $f(x^{j+1}) < f(x^j)$, and the numerical experiments in Sec.\ \ref{sec:numerical_results} (cf.\ Fig.\ \ref{fig:example_halfhalf}(b)) will indeed show that it is not a descent method. The following theorem shows that if $x^* \in \Bcl_{\eps_1}(x^1)$ and $\eps_1$ is small enough, then \algLocal{} is well-defined (i.e., Oracle \ref{oracle:1} is never called outside $U$) and the sequence $(x^j)_j$ generated by this method converges to $x^*$ with an R-superlinear rate:

    \begin{theorem} \label{thm:local_method_convergence}
        Assume that $f : U \rightarrow \R$ satisfies \ref{assum:A2} for $q, p \in \N$ and that $q \geq p$.
        Let $\sigma, \kappa \in (0,1)$.
        Then there is some $\epsmax > 0$ such that for all $\eps_1 \in (0,\epsmax]$ and all $x^1 \in \Bcl_{\eps_1}(x^*)$, \algLocal{} generates a sequence $(x^j)_j$ with
        \begin{align} \label{eq:local_method_convergence}
            \| x^j - x^* \| < \eps_j
            \quad \text{and} \quad
            \| x^j - x^{j+1} \| < \eps_j \quad \forall j \geq 2.
        \end{align}
        In particular, $(x^j)_j$ converges R-superlinearly to $x^*$ with order $(q+\sigma)/p$.
    \end{theorem}
    \begin{proof}
        Let $\epsmax \leq 1$ small enough so that $\Bcl_{2 \epsmax}(x^*) \subseteq U$. Let $\eps_1 \in (0,\epsmax]$ and $x^1 \in \Bcl_{\eps_1}(x^*)$. By construction of \algApproxW{}, $W_j$ satisfies \eqref{eq:T_W_error_estimate_2} (with $\eps = \eps_j$) for all $j \in \N$. Thus, Lem.\ \ref{lem:T_W_error_estimate}(b) shows that there is some $K \geq 0$ with
        \begin{align*}
            \| x^2 - x^* \|
            &= \| \bar{z}^{q,W_1}(x^1,\eps_1) - x^* \|
            \stackrel{\eqref{eq:T_W_error_estimate_3}}{\leq} \left( \frac{1 + K \eps_1^{1 - \sigma}}{\beta} \right)^{1/p} \eps_1^{\frac{q+\sigma}{p}}
            \leq \left( \frac{1 + K}{\beta} \right)^{1/p} \eps_1^{\frac{q+\sigma}{p}}.
        \end{align*}
        By Lem.\ \ref{lem:eps_estimate}(b) (for $M = ((1+K)/\beta)^{1/p}$), we can assume w.l.o.g.\ that $\epsmax$ is small enough to have $\| x^2 - x^* \| < \eps_2$, i.e., $x^2 \in \Bcl_{\eps_2}(x^*)$. Since $K$ only depends on $\epsmax$ (cf.\ Lem.\ \ref{lem:q_order_error_estimate}), induction shows that $\| x^j - x^* \| < \eps_j$ for all $j \in \N$, proving the first inequality in \eqref{eq:local_method_convergence} (and showing that $\Bcl_{\eps_j}(x^j) \subseteq \Bcl_{2 \epsmax}(x^*) \subseteq U$ for all $j \in \N$). For the second inequality, let $j \geq 2$. Then
        \begin{equation} \label{eq:proof_thm_local_method_convergence}
            \begin{aligned}
                \| x^j - x^{j+1} \|
                &= \| x^j - \bar{z}^{q,W_j}(x^j,\eps_j) \| \\
                &\leq \| x^j - x^* \| + \| x^* - \bar{z}^{q,W_j}(x^j,\eps_j) \| \\
                &= \| \bar{z}^{q,W_{j-1}}(x^{j-1},\eps_{j-1}) - x^* \| + \| \bar{z}^{q,W_j}(x^j,\eps_j) - x^* \| \\
                &\stackrel{\eqref{eq:T_W_error_estimate_3}}{\leq} \left( \frac{1 + K \eps_{j-1}^{1 - \sigma}}{\beta} \right)^{1/p} \eps_{j-1}^{\frac{q+\sigma}{p}} + \left( \frac{1 + K \eps_j^{1 - \sigma}}{\beta} \right)^{1/p} \eps_j^{\frac{q+\sigma}{p}} \\
                &\leq \left( \frac{1 + K}{\beta} \right)^{1/p} \eps_{j-1}^{\frac{q+\sigma}{p}} + \left( \frac{1 + K}{\beta} \right)^{1/p} \eps_j^{\frac{q+\sigma}{p}}.
            \end{aligned}
        \end{equation}
        By the second inequality in Lem.\ \ref{lem:eps_estimate}(b), we can assume w.l.o.g.\ that $\epsmax$ is small enough so that the right-hand side of \eqref{eq:proof_thm_local_method_convergence} is less than $\eps_j$. Finally, the order of convergence of $(x^j)_j$ follows from combination of the first inequality in \eqref{eq:local_method_convergence} with Lem.\ \ref{lem:eps_estimate}(a), completing the proof.
    \end{proof}

    The second inequality in \eqref{eq:local_method_convergence} shows that the $\eps$-ball constraint in the subproblem \eqref{eq:bar_z_epigraph} becomes inactive after the first iteration of \algLocal{} if $\eps_1$ is small enough and $x^* \in \Bcl_{\eps_1}(x^1)$. This behavior is analogous to the local convergence of the trust-region Newton method, where the trust region eventually becomes inactive when the method is close enough to the minimum (see, e.g., \cite{NW2006}, Thm.\ 4.9). This property will be crucial for the globalization of the local method in Sec.\ \ref{sec:globalization} (and \cite{GU2026b}).

    Note that Thm.\ \ref{thm:local_method_convergence} provides a rate of convergence with respect to the index $j$ in \algLocal{} (which can be interpreted as the ``serious steps'', yielding a result as in \cite{ASS2023}), but not with respect to the number of oracle calls. Since all oracle calls in \algLocal{} are performed during the execution of \algApproxW{} in Step \ref{state:local_method_approx_W}, a rate with respect to oracle calls can be obtained when the number of iterations in \algApproxW{} is bounded:
    \begin{corollary} \label{cor:N_step_convergence}
        In the setting of Thm.\ \ref{thm:local_method_convergence}, let $j(l)$ be the index $j$ of the iteration of \algLocal{} in which the $l$-th oracle call occurs (within \algApproxW{}). Assume that the number of oracle calls performed during each execution of \algApproxW{} is bounded by $N \in \N$. Then $(x^{j(l)})_l$ converges $N$-step R-superlinearly (cf.\ \cite{NW2006}, (5.51)) to $x^*$ with order $(q+\sigma)/p$. Furthermore, $(x^{j(l)})_l$ converges R-superlinearly with order $((q + \sigma) / p)^{1/N}$.
    \end{corollary}
    \begin{proof}
        \textbf{Part 1:} By construction, it holds $\| x^{j(l)} - x^* \| \leq \eps_{j(l)}$ and $\eps_{j(l+N)} \leq \eps_{j(l)+1}$ for all $l \in \N$. For $a > 0$ this implies
        \begin{align*}
            \frac{\eps_{j(l+N)}}{\eps_{j(l)}^a}
            \leq \frac{\eps_{j(l)+1}}{\eps_{j(l)}^a}
            \quad \forall l \in \N.
        \end{align*}
        As in the proof of Lem.\ \ref{lem:eps_estimate}(a), the $N$-step Q-superlinear convergence of $(\eps_{j(l)})_l$ with order $(q+\sigma)/p$ follows. In particular, $(x^{j(l)})_l$ converges $N$-step R-superlinearly with the same order. \\
        \textbf{Part 2:} To see that $(x^{j(l)})_l$ also converges R-superlinearly, note that by assumption, we have 
        \begin{align*}
            j(l) 
            \geq \lfloor \frac{l-1}{N} \rfloor + 1
            \geq \frac{l}{N}
            \quad \forall l \in \N.
        \end{align*}
        With a slight abuse of notation, consider the function $\eps : \R \rightarrow \R$, $j \mapsto \eps_1 \kappa^{(\frac{q+\sigma}{p})^{j - 1} - 1}$. Analogous to the proof of Lem.\ \ref{lem:eps_estimate}(a), it can be shown that the function $\eps$ decreases monotonically, which implies that $\eps_{j(l)} = \eps(j(l)) \leq \eps(l/N)$ for all $l \in \N$, and that $(\eps(l/N))_l$ vanishes Q-superlinearly with order $((q + \sigma) / p)^{1/N}$, which completes the proof. 
    \end{proof}

    Combined with Lem.\ \ref{lem:algo_approx_W_termination}(b), the previous corollary shows that if $S$ is finite (i.e., if $f$ is finite max-type function) and $\eps_1$ is small enough, then \algLocal{} generates a sequence that converges to $x^*$ at an R-superlinear rate with respect to oracle calls. (However, note that due to taking the $N$-th root, the order may be close to $1$.) The proof of Lem.\ \ref{lem:algo_approx_W_termination}(b) requires that $\sigma < 1$ in \algApproxW{}, so we cannot provably achieve an order of convergence of $(q+1)/p$ for $(x^j)_j$ while having bounded oracle calls in \algApproxW{}. In particular, for $q = p = 2$ and smooth $f$, we cannot recover the order $2$ of local convergence of the trust-region Newton method. (The difference is that the models in our method may be centered at any point in the trust region, not just its midpoint. This is also the reason why the trust-region radius must vanish in our method.)

\section{Globalization} \label{sec:globalization}

    In the previous section, we have resolved the Challenges \ref{enum:C1} and \ref{enum:C2}. To overcome Challenge \ref{enum:C3}, \algLocal{} has to be globalized. To do so, the idea is to construct an auxiliary trust-region method that generates sequences $(\hx^j)_j \subseteq \R^n$ and $(\glrad_j)_j \subseteq \R^{>0}$ with $\glrad_j \rightarrow 0$ and $\hx^j \in \Bcl_{\glrad_j}(x^*)$ for infinitely many $j$, while applying \algLocal{} with $x^1 = \hx^j$ and $\eps_1 = \glrad_j$ for each $j$. By Thm.\ \ref{thm:local_method_convergence}, this will eventually lead to a run of \algLocal{} being successful, in the sense that it generates a sequence converging R-superlinearly to $x^*$. To make this idea implementable, we have to provide a method that is able to generate $(\hx^j)_j$ and $(\glrad_j)_j$ with the above properties, and a way to check whether an application of \algLocal{} will be successful (since \algLocal{} has no stopping criterion). In this work, we only provide the latter, in Subsec.\ \ref{subsec:detecting_superlinear}. Since the method that provides $(\hx^j)_j$ and $(\glrad_j)_j$ must be a globally convergent solution method for nonsmooth optimization problems in its own right, its construction and the verification of the stated properties require their own theory. As such, we only give a brief summary of this method in Subsec.\ \ref{subsec:initial_data}, and refer to the accompanying paper \cite{GU2026b} for the details.

    \subsection{Detecting superlinear convergence} \label{subsec:detecting_superlinear}
    
        To obtain a criterion for a successful application of \algLocal{}, we exploit the second inequality in \eqref{eq:local_method_convergence}. It states that from the second iteration onward, the trust-region constraint in the subproblem \eqref{eq:bar_z_epigraph} that yields the next iterate is always inactive. As such, if this constraint is active for any iteration after the first one, we can immediately stop the algorithm. What remains is the question whether the trust-region constraint being inactive is sufficient for R-superlinear convergence to $x^*$. While we cannot prove that it is sufficient for convergence to $x^*$, it turns out that it is indeed sufficient for R-superlinear convergence to some critical point of $f$, which we prove in the following two lemmas. The first lemma shows that \emph{any} sequence $(x^j)_j \subseteq \R^n$ with $\| x^j - x^{j+1} \| \leq \eps_j$ for all $j \in \N$ and $(\eps_j)_j$ as in \eqref{eq:def_eps_j} converges R-superlinearly to \emph{some} point $\bar{x} \in \R^n$.
        
        \begin{lemma} \label{lem:global_convergence_local_algo}
            For $q,p \in \N$, $q \geq p$, $\eps_1 > 0$, and $\sigma, \kappa \in (0,1)$, consider the sequence $(\eps_j)_j$ from \eqref{eq:def_eps_j}, i.e., $\eps_j = \eps_1 \kappa^{(\frac{q+\sigma}{p})^{j-1} - 1}$ for $j \in \N$. If $(x^j)_j \subseteq \R^n$ is a sequence with $\| x^j - x^{j+1} \| \leq \eps_j$ for all $j \in \N$, then there are $C > 0$ and $\bar{x} \in \R^n$ such that
            \begin{align*}
                \| x^j - \bar{x} \| \leq C \eps_j \quad \forall j \in \N.
            \end{align*}
            In particular, $(x^j)_j$ converges R-superlinearly with order $(q + \sigma)/p$.
        \end{lemma}
        \begin{proof}
            For ease of notation let $Q := \frac{q + \sigma}{p}$, so $\eps_j = \eps_1 \kappa^{Q^{j-1} - 1}$. Since $q \geq p$ and $\sigma \in (0,1)$ it holds $Q > 1$. \\
            \textbf{Part 1:} For $j \in \N$ consider the sequence $(E_j)_j$ defined by $E_j := \sum_{i = j}^\infty \eps_i$.
            Since $Q > 1$, there is some $N > 0$ such that $\eps_j = \eps_1 \kappa^{Q^{j-1} - 1} < \eps_1 \kappa^j$ for all $j > N$. This means that $(\eps_j)_j$ eventually decreases faster than a geometric sequence, which implies that $E_j$ is finite for all $j \in \N$ and that $E_j \rightarrow 0$. \\
            \textbf{Part 2:} For $j,k \in \N$, $j \leq k$, the triangle inequality implies
            \begin{align*}
                \| x^j - x^k \|
                \leq \sum_{i = j}^{k-1} \| x^i - x^{i+1} \|
                \leq \sum_{i = j}^{k-1} \eps_i 
                < E_j.
            \end{align*}
            Since $(E_j)_j$ vanishes, this shows that $(x^j)_j$ is a Cauchy sequence, which implies that it has a limit $\bar{x} \in \R^n$. In particular, letting $k \rightarrow \infty$ yields $\| x^j - \bar{x} \| \leq E_j$ for all $j \in \N$. \\
            \textbf{Part 3:} For all $j \in \N$ it holds
            \begin{align*}
                \frac{E_j}{\eps_j}
                &= \frac{\sum_{i = j}^\infty \eps_i}{\eps_j}
                = \sum_{i=j}^\infty \frac{\eps_i}{\eps_j}
                = \sum_{i=j}^\infty \frac{\kappa^{Q^{i-1} - 1}}{\kappa^{Q^{j-1} - 1}}
                = \sum_{i=j}^\infty \kappa^{Q^{i-1} - Q^{j-1}}
                = \sum_{i=j}^\infty \kappa^{Q^{j-1} (Q^{i-j} - 1)} \\
                &= \sum_{i=j}^\infty \left( \kappa^{Q^{j-1}} \right)^{Q^{i-j} - 1}
                = \sum_{l = 0}^\infty \left( \kappa^{Q^{j-1}} \right)^{Q^l - 1}
                \leq \sum_{l = 0}^\infty \kappa^{Q^l - 1}
                =: C \in \R,
            \end{align*}
            where finiteness of $C$ follows from $(\kappa^{Q^l - 1})_l$ eventually decreasing faster than $(\kappa^l)_l$. Combined with Part 2 we obtain
            \begin{align*}
                \| x^j - \bar{x} \| 
                \leq E_j
                \leq C \eps_j.
            \end{align*}
            Finally, the order of convergence of $(x^j)_j$ follows from Lem.\ \ref{lem:eps_estimate}(a).
        \end{proof}
    
        The second lemma shows that if the trust-region constraint in \algLocal{} is inactive for all $j$ larger than some $\jthr \in \N$, then $(x^j)_j$ converges R-superlinearly to a point that is at least critical. 
    
        \begin{lemma} \label{lem:global_convergence_critical}
            Let $q, p \in \N$. Assume that $f : U \rightarrow \R$ satisfies \ref{assum:A1} and that \algLocal{} generates a sequence $(x^j)_j$ with $\Bcl_{\eps_j}(x^j) \subseteq U$ for all $j \in \N$. If there is some $\jthr \in \N$ such that $\| x^j - x^{j+1} \| < \eps_j$ for all $j > \jthr$ (i.e., the trust-region constraint in \eqref{eq:bar_z_epigraph} is inactive for all $j > \jthr$), then $(x^j)_j$ converges R-superlinearly with order $(q+\sigma)/p$ to a critical point of $f$.
        \end{lemma}
        \begin{proof}
            \textbf{Part 1:} We first consider the optimality conditions of the epigraph formulation \eqref{eq:bar_z_epigraph} for general $x$, $\eps$, and $W$. To this end, let $x \in U$ and $\eps > 0$ with $\Bcl_\eps(x) \subseteq U$ and let $W \subseteq \Bcl_{\eps}(x)$ be finite and nonempty. Denote $\bar{z}^W = \bar{z}^{q,W}(x,\eps)$. It is easy to see that the constraints in \eqref{eq:bar_z_epigraph} satisfy the MFCQ (see, e.g., \cite{NW2006}, Def.\ 12.6). Assume that the trust-region constraint is inactive. Then the first-order necessary optimality conditions imply that there are $\lambda_y \geq 0$, $y \in W$, such that
            \begin{align*}
                0  
                = \begin{pmatrix}
                    0 \\
                    1
                \end{pmatrix}
                +
                \sum_{y \in W} \lambda_y
                \begin{pmatrix}
                    \nabla (T^q f_{s(y)}(\cdot,y))(\bar{z}^W) \\
                    -1
                \end{pmatrix}
            \end{align*}
            with $\lambda_y = 0$ if $T^q f_{s(y)}(\bar{z}^W,y) < \T^{q,W}(\bar{z}^W)$. The second line of this equation yields $\sum_{y \in W} \lambda_y = 1$ and the first line yields
            \begin{equation} \label{eq:proof_lem_global_convergence_critical_1}
                \begin{aligned}
                    0
                    &= \sum_{y \in W} \lambda_y \nabla (T^q f_{s(y)}(\cdot,y))(\bar{z}^W) \\
                    &= \underbrace{\sum_{y \in W} \lambda_y \nabla f_{s(y)}(y)}_{(a)}
                    +
                    \underbrace{\sum_{y \in W} \lambda_y  \sum_{m = 2}^q \frac{1}{m!} \nabla \left( \deriv^{m} f_{s(y)}(y)(\cdot - y)^m \right)(\bar{z}^W)}_{(b)}.
                \end{aligned}
            \end{equation}
            \textbf{Part 2:} For $j > \jthr$ consider \eqref{eq:proof_lem_global_convergence_critical_1} with $x = x^j$, $\eps = \eps_j$, and $W = W_j$ (cf.\ Step \ref{state:approx_W_stopping_criterion} in \algLocal{}). Note that for each $y \in W_j$ and $m \in \{2,\dots,q\}$, each summand in $\nabla \left( \deriv^{m} f_{s(y)}(y)(\cdot - y)^m \right)(\bar{z}^{W_j})$ contains a factor $\bar{z}_i^{W_j} - y_i$, $i \in \{1,\dots,n\}$ (cf.\ \eqref{eq:def_taylor_series}). Since $\| \bar{z}^{W_j} - y \| \leq 2 \eps_j \rightarrow 0$ and the partial derivatives are bounded above, this implies that the term (b) in \eqref{eq:proof_lem_global_convergence_critical_1} vanishes for $j \rightarrow \infty$. In particular, since the left-hand side in \eqref{eq:proof_lem_global_convergence_critical_1} is zero, the term (a) vanishes as well. \\
            \textbf{Part 3:} By Lem.\ \ref{lem:global_convergence_local_algo} the sequence $(x^j)_j$ converges to some $\bar{x} \in \R^n$ (R-superlinearly with order $(q+\sigma)/p$). Note that $\nabla f_{s(y)}(y) \in \partial f(y) \subseteq \conv(\partial f(\Bcl_{\eps_j}(x^j))) =: \partial_{\eps_j} f(x^j)$ for all $y \in W_j$. (The set $\partial_{\eps_j} f(x^j)$ is known as the \emph{Goldstein} $\eps_j$\emph{-subdifferential} \cite{G1977} of $f$ at $x^j$.) Part 2 showed that the element with the smallest norm in $\partial_{\eps_j} f(x^j)$ vanishes for $j \rightarrow \infty$. Upper semicontinuity of the Clarke subdifferential implies that $\bar{x}$ is critical (see, e.g., \cite{G2022}, Def.\ 4.4.1 and Lem.\ 4.4.4 for details), completing the proof. 
        \end{proof}
    
        Note that Lem.\ \ref{lem:global_convergence_critical} does not require that $f$ satisfies the growth assumption \ref{assum:A2} and can therefore be used to detect R-superlinear convergence of $(x^j)_j$ to critical points for a relatively general class of functions.

    \subsection{Computing the initial data} \label{subsec:initial_data}

        In this subsection we briefly summarize the method from \cite{GU2026b} which, given a sequence $(\glrad_j)_j$ with $\glrad_j \rightarrow 0$, is able to compute a sequence $(\hx^j)_j$ with $\hx^j \in \Bcl_{\glrad_j}(x^*)$ for infinitely many $j \in \N$. Its derivation is based around the theoretical quantity
        \begin{align*}
            \Lambda^p(x,\glrad) := \frac{f(x) - f(z^*(x,\glrad))}{\glrad^p} \geq 0,
            \ \text{where} \
            z^*(x,\glrad) \in \argmin_{z \in \Bcl_\glrad(x)} f(x),
        \end{align*}
        $p \in \N$, $x \in U$, and $\glrad > 0$ with $\Bcl_\glrad(x) \subseteq U$. The idea is to show that for a function satisfying \ref{assum:A2} for order $p$, there is a constant $C > 0$ such that if $x$ is close to $x^*$ but $x^* \notin \Bcl_\glrad(x)$, then $\Lambda^p(x,\glrad) \geq C$. In words, this means that as long as $x^*$ is not in the trust region $\Bcl_\glrad(x)$, the value of $f$ can be decreased by at least $C \glrad^p$. For example, consider the function $f : \R \rightarrow \R$, $x \mapsto a |x|^p$ for $a > 0$. Let $x, \glrad > 0$ such that $0 \notin \Bcl_\glrad(x)$ (i.e., $\glrad < x$). Then $z^*(x,\glrad) = x - \glrad$ and, since $(x-\glrad)/x \in (0,1)$, we have
        \begin{align*}
            \Lambda^p(x,\glrad)
            = \frac{a x^p - a (x - \glrad)^p}{\glrad^p}
            = a \frac{1 - (\frac{x - \glrad}{x})^p}{(1 - \frac{x - \glrad}{x})^p}
            \geq a \frac{1 - \frac{x - \glrad}{x}}{(1 - \frac{x - \glrad}{x})^p}
            \geq a.
        \end{align*}
        In \cite{GU2026b}, for general functions satisfying \ref{assum:A2}, it is shown that for $p = 1$, the above property holds when $S$ is finite (cf.\ \cite{GU2026b}, Sec.\ 4.1), and for $p = 2$, it holds when $S$ is finite and the vanishing convex combination of gradients at $x^*$ is unique and ``stable'' (cf.\ \cite{GU2026b}, Sec.\ 4.2).
        
        Now consider a sequence $(\glrad_j)_j$ with $\glrad_j \rightarrow 0$. From a theoretical point of view, the above property of $\Lambda^p$ allows for the conceptual \algGlobal{} for computing a corresponding sequence $(\hx^j)_j$.
        \begin{algorithm} 
            \caption{Conceptual globalized method}
            \label{algo:conceptual_global_method}
            \begin{algorithmic}[1] 
                \Require Initial point $\hx^0 = \hx^{1,0} \in \R^n$, vanishing sequences $(\glrad_j)_j$, $(\tau_j)_j \subseteq \R^{> 0}$, growth order $p \in \N$.
                \For{$j = 1,2,\dots$}
                    \For{$i = 0,1,\dots$}
                        \If{$\Lambda^p(\hx^{j,i},\glrad_j) < \tau_j$} \label{state:conceptual_global_method_decrease_condition}
                            \State Break $i$-loop.
                        \Else
                            \State Set $\hx^{j,i+1} = z^*(\hx^{j,i},\glrad_j)$. \label{state:conceptual_global_method_new_iterate}
                        \EndIf
                    \EndFor
                    \State Set $\hx^{j+1,0} = \hx^{j,i}$ and $\hx^j = \hx^{j,i}$.
                    \State Apply \algLocal{} with $x^1 = \hx^j$ and $\eps_1 = \glrad_j$. If the trust-region constraint is 
                    \Statex \hspace{\algorithmicindent}active in any iteration after the first, then stop \algLocal{}. \label{state:conceptual_global_method_apply_local}
                \EndFor
            \end{algorithmic}
        \end{algorithm}
        For each $j$, it decreases the objective value by $\tau_j \glrad_j^p$ as long as possible. When this is no longer possible, the trust-region radius is changed and \algLocal{} is attempted. A proof by contradiction shows that this eventually leads to a successful run of \algLocal{}: If \algGlobal{} remains in Step \ref{state:conceptual_global_method_apply_local} infinitely, then \algLocal{} is successful by Lem.\ \ref{lem:global_convergence_local_algo} and Lem.\ \ref{lem:global_convergence_critical}. Assume that this never happens. If $f$ is bounded below, then the $i$-loops must always be finite, such that $(\hx^j)_j$ is an infinite sequence. In particular, $\Lambda^p(\hx^j,\glrad_j) \rightarrow 0$ for $j \rightarrow \infty$. Using the Goldstein $\eps$-subdifferential \cite{G1977}, one can show that this implies that all accumulation points of $(\hx^j)_j$ must be critical points of $f$ (cf.\ \cite{GU2026b}, Sec.\ 2.1). If $f$ satisfies \ref{assum:A2} for $x^*$ being one of these accumulation points, then the above property of $\Lambda^p$ assures that $x^* \in \Bcl_{\glrad_j}(\hx^j)$ for all $j$ with $\tau_j < C$, which are infinitely many since $\tau_j \rightarrow 0$. Thus, for some $j \in \N$, the requirements of Thm.\ \ref{thm:local_method_convergence} must hold, such that the algorithm remains in Step \ref{state:conceptual_global_method_apply_local} infinitely, leading to a contradiction. 
    
        Clearly \algGlobal{} is purely conceptual since $z^*$ (and therefore $\Lambda^p$) cannot be computed in practice. However, it can be turned into an implementable algorithm by replacing $z^*(\hx^{j,i},\glrad_j)$ in Step \ref{state:conceptual_global_method_decrease_condition} and Step \ref{state:conceptual_global_method_new_iterate} by $\bar{z}^{q,W}(\hx^{j,i},\glrad_j)$ from \eqref{eq:bar_z_epigraph} (for a set $W$ from \algApproxW{}). While the resulting method uses the same model and the same subproblem as \algLocal{}, the key difference is that \algGlobal{} enforces sufficient decrease in every iteration via Step \ref{state:conceptual_global_method_decrease_condition} and does not attempt to achieve fast convergence via Lem.\ \ref{lem:T_W_error_estimate}. In particular, the sequence $(\glrad_j)_j$ does not have to be chosen as in Challenge \ref{enum:C2}, and can instead be any vanishing sequence (like a linearly vanishing sequence as in standard trust-region methods). Fortunately, for $q \geq p$, this modified version of \algGlobal{} still retains all convergence properties discussed above (cf.\ \cite{GU2026b}, Cor.\ 3.1), such that it eventually executes a successful run of \algLocal{} with R-superlinear convergence in Step \ref{state:conceptual_global_method_apply_local}. Numerical experiments with the resulting method are shown in \cite{GU2026b}, Sec.\ 5.
    
    \section{Numerical experiments} \label{sec:numerical_results}

        In this section, we show the behavior of an implementation of \algLocal{} in numerical experiments. We first consider a univariate toy example that allows us to verify the order $(q + \sigma)/p$ of R-convergence from Thm.\ \ref{thm:local_method_convergence} for different model orders $q$. Afterwards, we analyze the behavior on a nonconvex finite max-type function and on a convex \lc{2} function which is not of finite max-type. Finally, we compare it to the superlinear solvers \vubundle{}\footnote{\url{https://github.com/GillesBareilles/NonSmoothSolvers.jl} (Retrieved Mar.\ 24, 2026)} and \superpolyak{}\footnote{\url{https://github.com/COR-OPT/SuperPolyak.jl} (Retrieved Mar.\ 24, 2026)}. Matlab code for the reproduction of all experiments shown in this section is available at \url{https://github.com/b-gebken/higher-order-trust-region-bundle-method}. 

        In all experiments, we assume that $f$ is sufficiently smooth at every point where Oracle \ref{oracle:1} is called, and use the exact analytic formulas of $f$ for the derivatives. For the parameters of \algLocal{}, we always use $\sigma = 0.5$ and $\kappa = 0.75$. As a stopping criterion, we check whether $\eps_j$ lies below a certain threshold value $\epsthr$. (This value varies in our experiments due to the different accuracies with which the subproblem \eqref{eq:bar_z_epigraph} is solved for different $q$.) For the initialization of \algApproxW{} (in Step \ref{state:local_method_approx_W} of \algLocal{}), except for Ex.\ \ref{example:1d_symbolic}, we reuse all points in the current trust region at which the oracle was already evaluated in previous iterations (cf.\ Rem.\ \ref{rem:bundle_initialization}(a)). (For Ex.\ \ref{example:1d_symbolic}, we simply use $W^1 = \{ x \}$ in \algApproxW{}.) While Thm.\ \ref{thm:local_method_convergence} only guarantees local convergence of \algLocal{}, we deliberately do not choose the initial points $x^1$ particularly close to the respective minima $x^*$ to highlight the surprising robustness of \algLocal{} when it comes to the initial data. In particular, as suggested by Lem.\ \ref{lem:global_convergence_critical}, we will see that the first few iterates $x^j$ lying outside $\Bcl_{\eps_j}(x^*)$ may not cause any convergence issues. For the comparison to other solvers, we mainly focus on the number of oracle calls as a performance metric. For \algLocal{}, all derivatives are evaluated the same number of times (during the construction of the subproblem \eqref{eq:bar_z_epigraph} in Step \ref{state:approx_W_solve_subproblem} of \algApproxW{}). The number of objective evaluations is larger by one, since the objective value of the final iterate is evaluated in Step \ref{state:approx_W_stopping_criterion} of \algApproxW{} before the algorithm stops. To show the impact of the effort of solving subproblem \eqref{eq:bar_z_epigraph}, we also state the actual runtimes for the comparison\footnote{Hardware used for the experiments: Intel(R) Core(TM) i7-8565 CPU@1.80GHz, Intel(R) UHD Graphics 620, 16GB RAM.}. However, since the methods are implemented in different programming languages, and oracle calls are cheap in our examples, comparing these runtimes has limited significance. 

        Clearly, \algLocal{} can only approximate the minimum of a function up to the accuracy with which the subproblem \eqref{eq:bar_z_epigraph} is solved. This introduces unwanted artifacts into \algLocal{} when $\eps_j$ becomes lower than that accuracy. In particular, $\eps_j$ may not lie below machine precision (which, in Matlab, is $2^{-52} \approx 2 \cdot 10^{-16}$). While this is unlikely to be an issue in practice, it does become a hindrance when analyzing high orders of convergence. As such, to first show the ``clean'' behavior of \algLocal{}, we consider an example with $n = 1$ and highly accurate solutions of \eqref{eq:bar_z_epigraph}. For $n = 1$, the solution of \eqref{eq:bar_z_epigraph} is a critical point of one of the Taylor expansions, a point where two expansions have the same value, or one of the two boundary points of $\Bcl_\eps(x)$. Since these are (typically) finitely many points, we can simply check their objective values to find the solution. To achieve high accuracy, we use Matlab's variable precision arithmetic (\texttt{vpa}).
        
        \begin{example} \label{example:1d_symbolic}
            For $n \in \N$ consider the nonconvex function
            \begin{align*}
                f : \R^n \rightarrow \R, \quad x \mapsto \max_{i \in \{1,\dots,n\}} \sqrt{|x_i| + 1/4} - 1/2.
            \end{align*}
            It is easy to see that $f$ is a finite max-type function with $|S| = 2n$. The unique global minimum is $x^* = 0 \in \R^n$, for which the growth assumption \ref{assum:A2} holds for $p = 1$. Now consider the case $n = 1$. Fig.\ \ref{fig:example_1d_symbolic_and_LW2019_85_I}(a) shows, in black, the distances $\| x^j - x^* \|$ for sequences $(x^j)_j$ generated by \algLocal{} for $q \in \{1,\dots,5\}$, $x^1 = 0.1$, $\eps_1 = 0.5$, and $\epsthr = 10^{-500}$. 
            \begin{figure}
                \centering
                \parbox[b]{0.49\textwidth}{
                    \centering 
                    \includegraphics[width=0.45\textwidth]{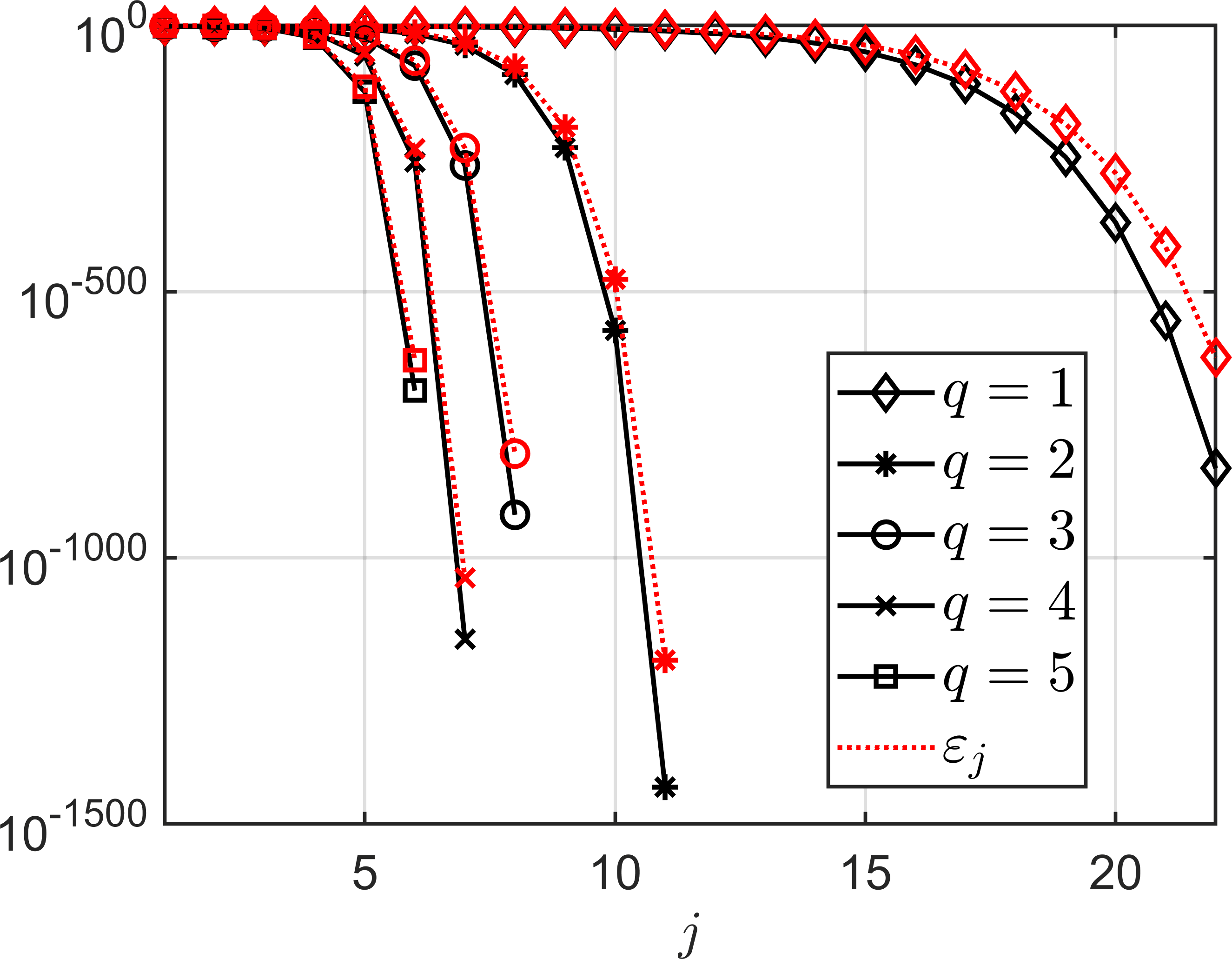}\\
                    (a)
        		}
                \parbox[b]{0.49\textwidth}{
                    \centering 
                    \includegraphics[width=0.425\textwidth]{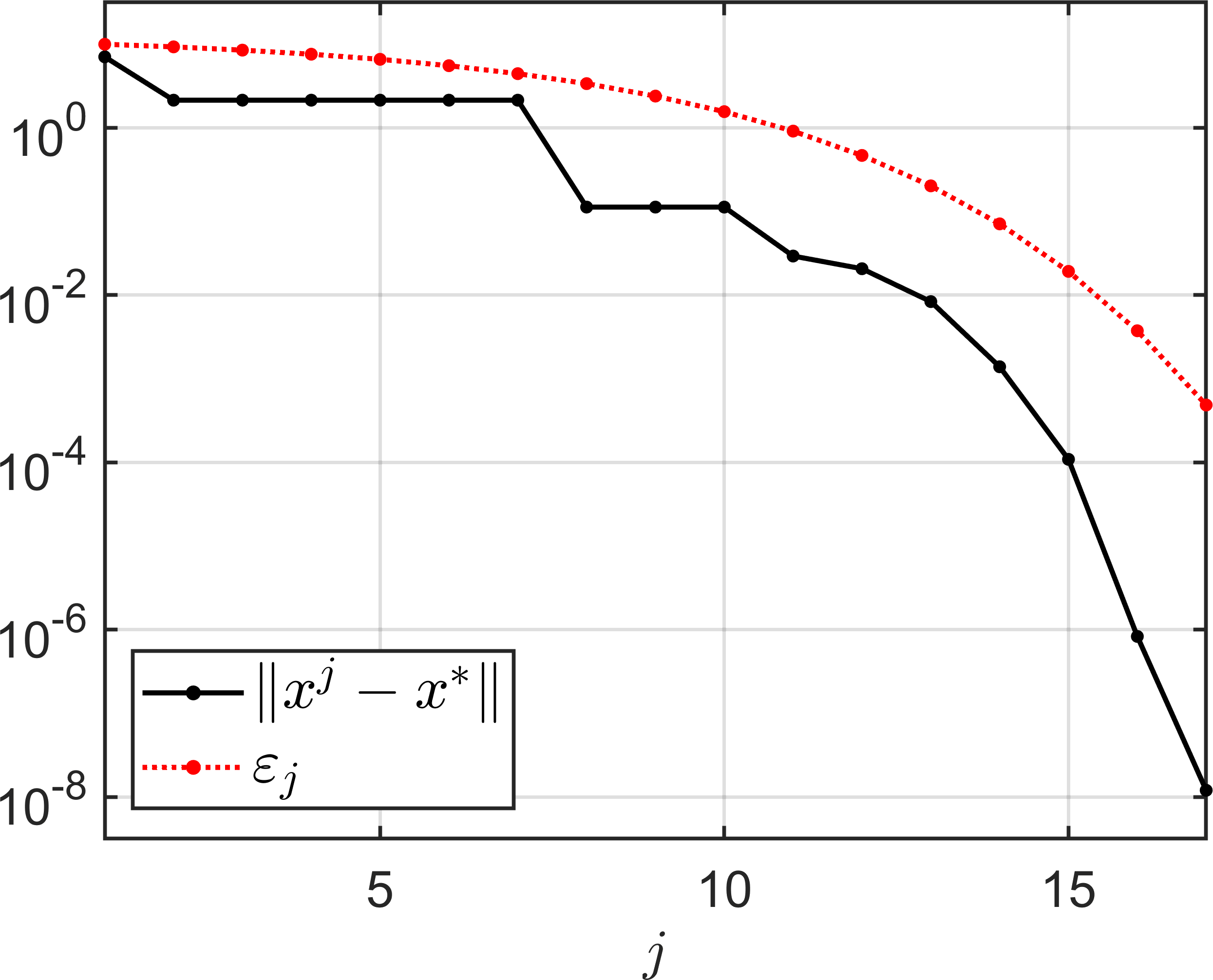}\\
                    (b)
        		}
                \caption{(a) The distance $\| x^j - x^* \|$ (black) for sequences $(x^j)_j$ generated by \algLocal{} with varying order $q$ of Taylor expansion in Ex.\ \ref{example:1d_symbolic}. The red, dotted lines show, depending on the marker, the corresponding upper bound $(\eps_j)_j$ from Thm.\ \ref{thm:local_method_convergence}. (b) The distance $\| x^j - x^* \|$ in Ex.\ \ref{example:LW2019_85} and the corresponding sequence $(\eps_j)_j$.}
                \label{fig:example_1d_symbolic_and_LW2019_85_I}
            \end{figure}
            The subproblems were solved with $2000$ digits of accuracy via Matlab's \texttt{vpa}. For each run, it holds $|W_j| = 2$ in every iteration. As expected due to Thm.\ \ref{thm:local_method_convergence}, the distance $\| x^j - x^* \|$ is bounded above by the corresponding sequence $(\eps_j)_j$, shown as red, dotted lines. In particular, since $p = 1$, the order $(q + \sigma)/p$ of R-convergence is quadratic (order $\geq 2$) for $q = 2$, cubic (order $\geq 3$) for $q = 3$, quartic (order $\geq 4$) for $q = 4$, and quintic (order $\geq 5$) for $q = 5$. 
        \end{example}

        Due to the simplicity of the nonsmoothness of the function in Ex.\ \ref{example:1d_symbolic} for $n = 1$, \algApproxW{} (with initialization $W^1 = \{ x \}$) only required two oracle calls in every iteration of \algLocal{}. The next example shows a more realistic case with a more complex nonsmooth structure. From now on, we always we use the bundling technique described in Rem.\ \ref{rem:bundle_initialization}(a) for initializing \algApproxW{}. We consider the quadratically growing, nonconvex function (8.5) from \cite{LW2019}, and use second-order models for \algLocal{}. The subproblem \eqref{eq:bar_z_epigraph} is solved via \ipopt{} \cite{WB2005} (using the Matlab interface \mexipopt{}\footnote{\url{https://github.com/ebertolazzi/mexIPOPT} (Retrieved Mar.\ 24, 2026)}). Since \ipopt{} failed to converge when $\eps_j \leq 10^{-4}$, we use the threshold $\epsthr = 10^{-3}$ for stopping.

        \begin{example} \label{example:LW2019_85}
            For $n \in \N$, $m \in \{1,\dots,n+1\}$, and $I = \{1,\dots,m\}$, consider the nonconvex function
            \begin{align*}
                f : \R^n \rightarrow \R,
                \quad
                x \mapsto \sum_{i \in I} \left| g_i^\top x + \frac{1}{2} x^\top H_i x + \frac{c_i}{24} \| x \|^4 \right|
            \end{align*}
            from \cite{LW2019}, where $c_i > 0$ for all $i \in I$, $H_i \in \R^{n \times n}$ is symmetric, pos.\ definite for all $i \in I$, and the vectors $g_i \in \R^n$, $i \in I$, are affinely independent with $\sum_{i \in I} \lambda_i g_i = 0$ for some $\lambda \in (\R^{>0})^m$ with $\sum_{i \in I} \lambda_i = 1$. It is easy to see that $f$ is a finite max-type function with $|S| = 2^m$. The global minimum of this function is $x^* = 0 \in \R^n$, for which the growth assumption \ref{assum:A2} holds with $p = 2$ (since all $H_i$ are pos.\ definite). 
            We generate a random instance of this problem for $n = 50$ and $m = 40$ and apply \algLocal{} with $q = 2$, $x^1 = (1,\dots,1)^\top \in \R^{50}$, $\eps_1 = 10$, and $\epsthr = 10^{-3}$. (For details on the random generation, see the corresponding code.) Fig.\ \ref{fig:example_1d_symbolic_and_LW2019_85_I}(b) shows the distance $\| x^j - x^* \|$ of the resulting sequence $(x^j)_j$ and the upper bound $(\eps_j)_j$, confirming the R-superlinear convergence (with order $(q + \sigma)/p = 1.25$). Fig.\ \ref{fig:example_LW2019_85_II}(a) shows the number of oracle calls that were required by \algApproxW{} in each iteration of \algLocal{}. 
            \begin{figure}
                \centering
                \parbox[b]{0.49\textwidth}{
                    \centering 
                    \includegraphics[width=0.45\textwidth]{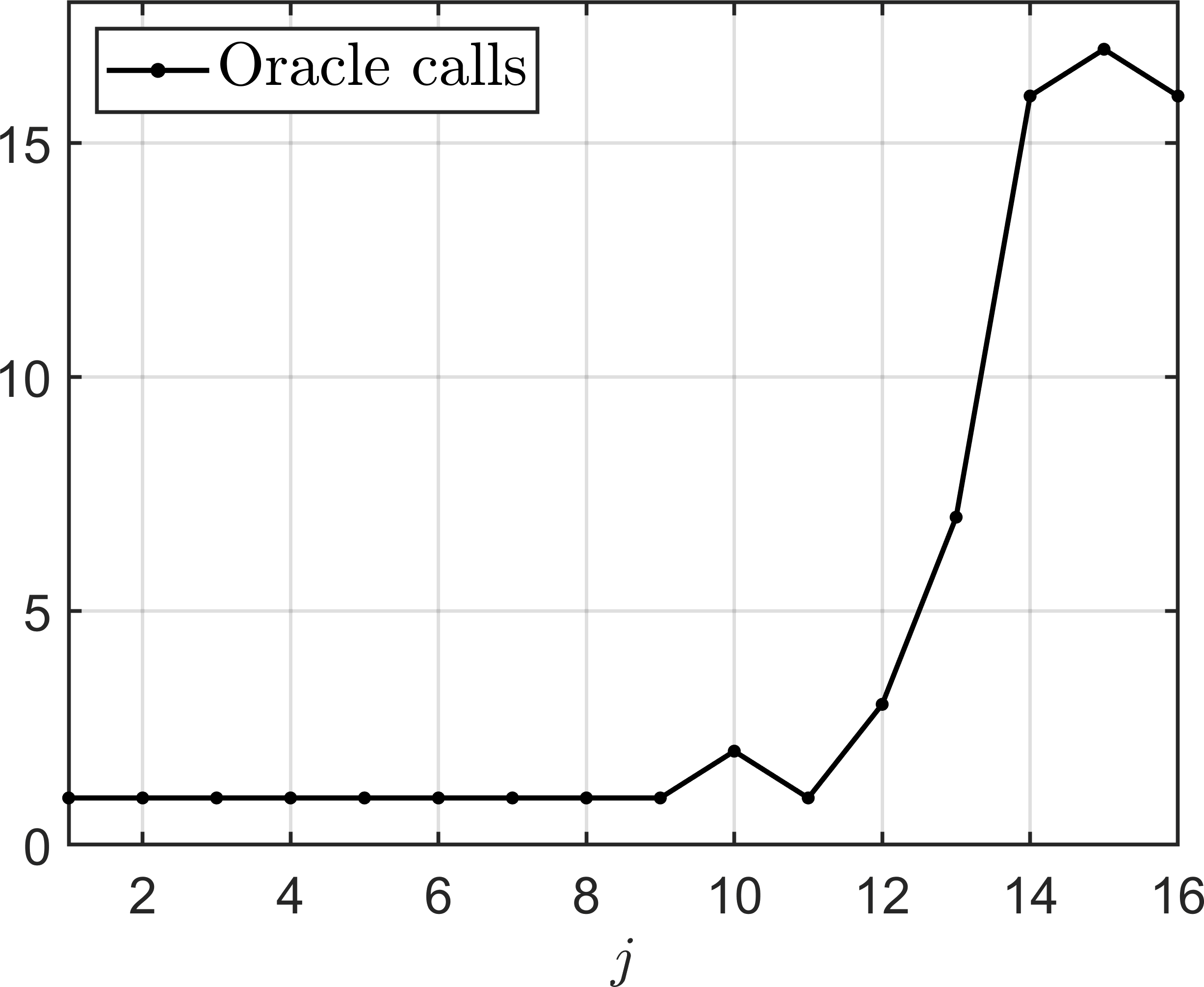}\\
                    (a)
        		}
                \parbox[b]{0.49\textwidth}{
                    \centering 
                    \includegraphics[width=0.45\textwidth]{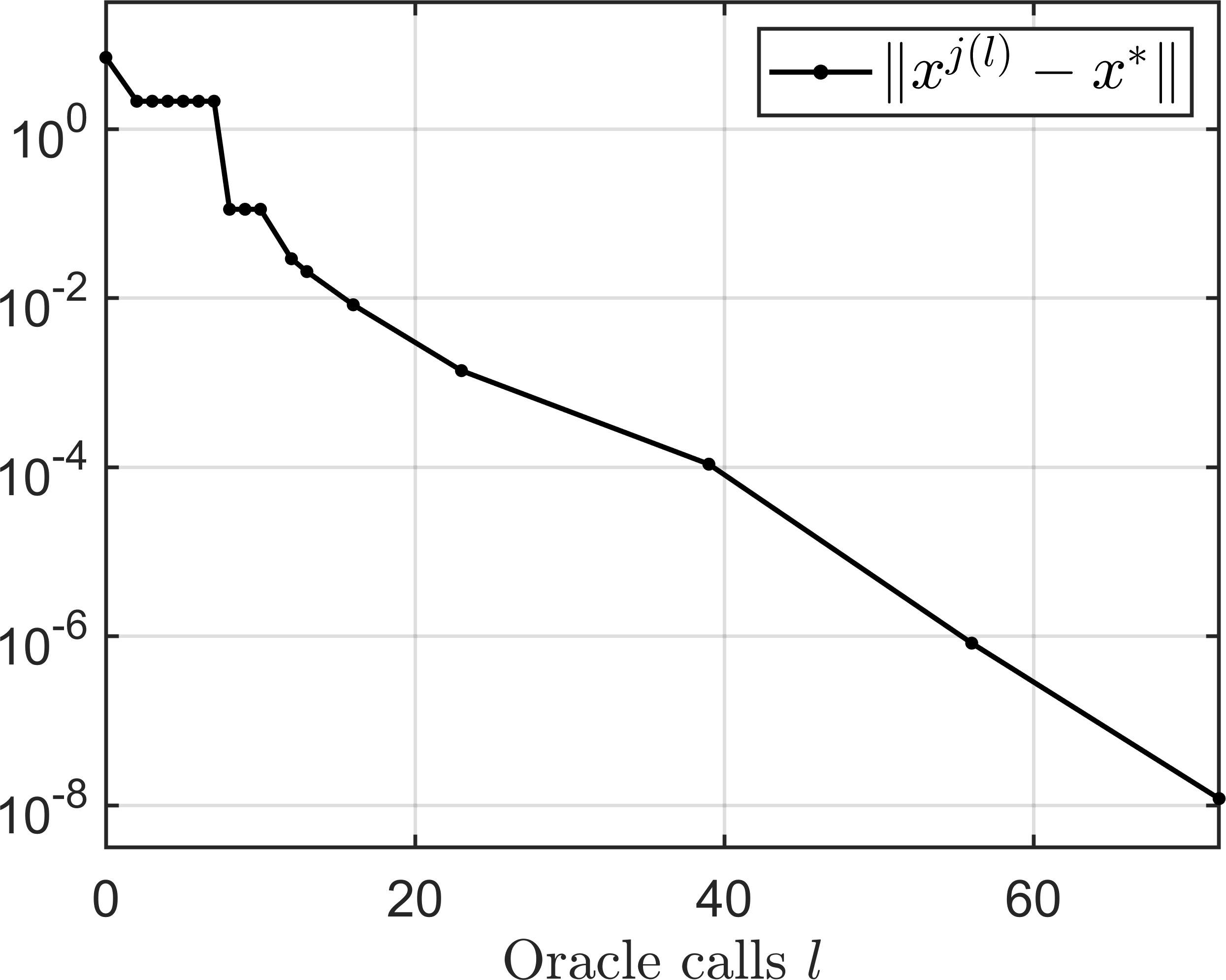}\\
                    (b)
        		}
                \caption{(a) The number of oracle calls required by \algApproxW{} in each iteration of \algLocal{} in Ex.\ \ref{example:LW2019_85}. (b) The distance $\| x^{j(l)} - x^* \|$ with $(x^{j(l)})_l$ as in Cor.\ \ref{cor:N_step_convergence}.}
                \label{fig:example_LW2019_85_II}
            \end{figure}
            We see that the closer $x^j$ is to the minimum, the more oracle calls are required, i.e., the larger the set $W_j$. Finally, Fig.\ \ref{fig:example_LW2019_85_II}(b) shows the speed of convergence with respect to oracle calls, i.e., it shows the distance $\| x^{j(l)} - x^* \|$ for the sequence $(x^{j(l)})_l$ from Cor.\ \ref{cor:N_step_convergence}, where $j(l)$ is the iteration of \algLocal{} in which the $l$-th oracle call occurred. (For simplicity, only the oracle calls where $j(l)$ changes are plotted.) Since we stop the algorithm already when $\eps_j \leq 10^{-3}$ (due to the accuracy of \ipopt{}), the R-superlinear convergence is not (yet) visible here.
        \end{example}

        In both examples considered so far, the objective was a finite max-type function. To show the behavior of \algLocal{} for \lc{2} functions that are not of finite max-type and for which no representation as in \eqref{eq:lower_Ck_representation} is practically available, we consider an example from the area of eigenvalue optimization \cite{O1992,FN1995,LW2019}, where the largest eigenvalue of an affine combination of matrices is minimized. Since $S$ is infinite in this case, we cannot use Lem.\ \ref{lem:algo_approx_W_termination}(b) to guarantee boundedness of the oracle calls in \algApproxW{} in Step \ref{state:local_method_approx_W} of \algLocal{}. Here, we also show a comparison to the \hanso{}\footnote{\url{https://cs.nyu.edu/~overton/software/hanso/} (Retrieved Mar.\ 24, 2026)} software package.

        \begin{example} \label{example:LW2019_eigval}
            For $n, m \in \N$ and symmetric matrices $A_0, \dots, A_n \in \R^{m \times m}$, consider the function
            \begin{align*}
                f : \R^n \rightarrow \R, \quad
                x \mapsto \lambda_{\text{max}}\left( A_0 + \sum_{i = 1}^m x_i A_i \right),
            \end{align*}
            where $\lambda_{\text{max}}(A)$ denotes the largest eigenvalue of a matrix $A$. This function is convex, and thus \lc{2} (cf.\ \cite{RW1998}, Thm.\ 10.33), but in general not of finite max-type (see \cite{O1992}, p.\ 89). It is bounded below if and only if there is no $x$ for which $\sum_{i = 1}^m x_i A_i$ is positive definite (cf.\ \cite{FN1995}, p.\ 227). We are not aware of results about the growth of $f$ around its minimum (if existing), and we blindly assume that it grows at least quadratically, i.e., with $p = 2$. We generate a random instance of this problem for $n = 50$ and $m = 25$ and apply \algLocal{} with $q = 2$, $x^1 = 0 \in \R^{50}$, $\eps^1 = 0.5$, and $\epsthr = 10^{-3}$. (For details on the random generation, see the corresponding code.) The subproblem \eqref{eq:bar_z_epigraph} is solved via \ipopt{}.
            Since an explicit expression for the minimum is not available for this example, we use \hanso{} (with starting point $x^1$ and default parameters) to compute a reference solution $\tilde{x}^*$. Fig.\ \ref{fig:example_LW2019_eigval}(a) shows the distance $\| x^j - \tilde{x}^* \|$ for the sequence $(x^j)_j$ generated by \algLocal{}.
            \begin{figure}
                \centering
                \parbox[b]{0.49\textwidth}{
                    \centering 
                    \includegraphics[width=0.45\textwidth]{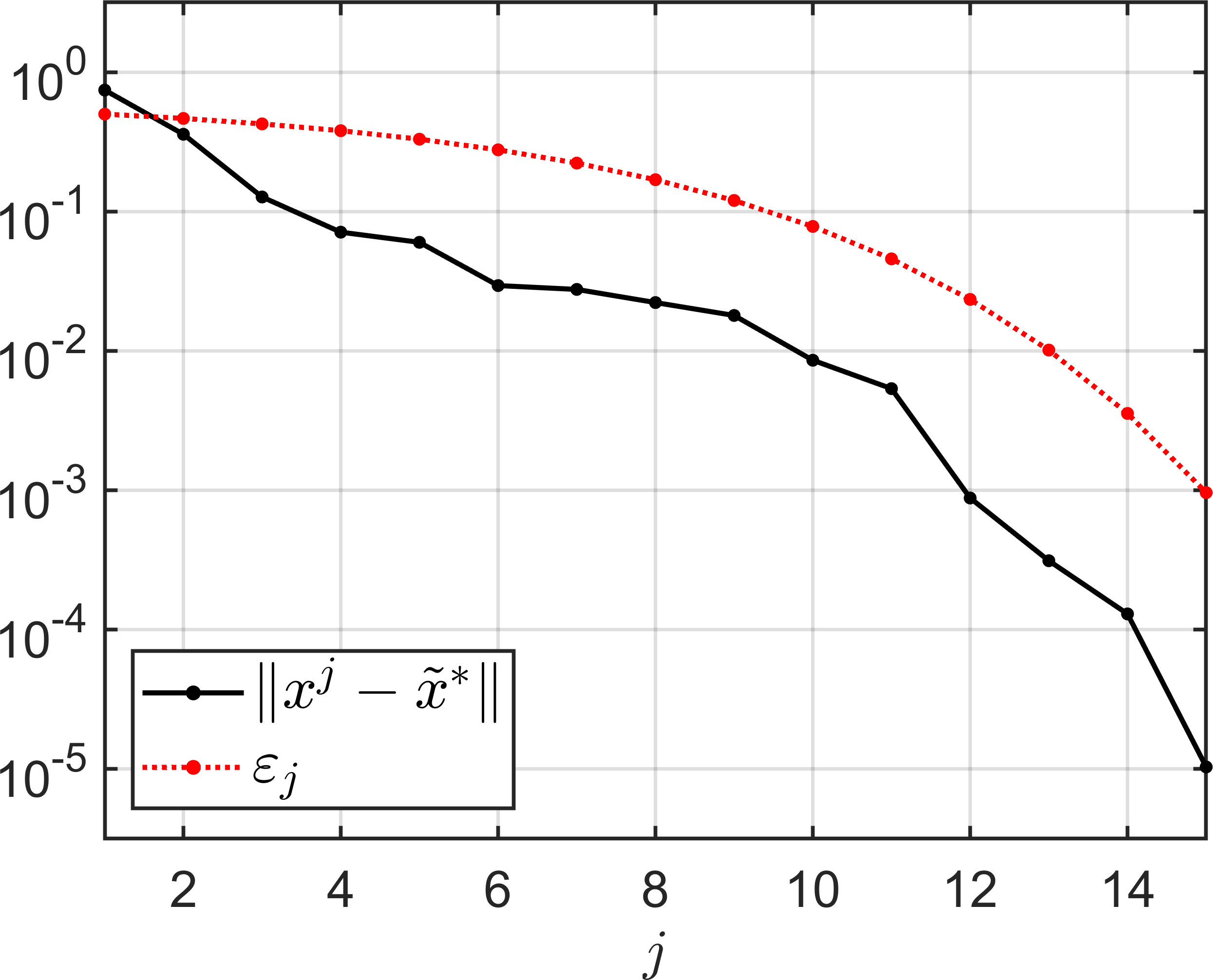}\\
                    (a)
        		}
                \parbox[b]{0.49\textwidth}{
                    \centering 
                    \includegraphics[width=0.45\textwidth]{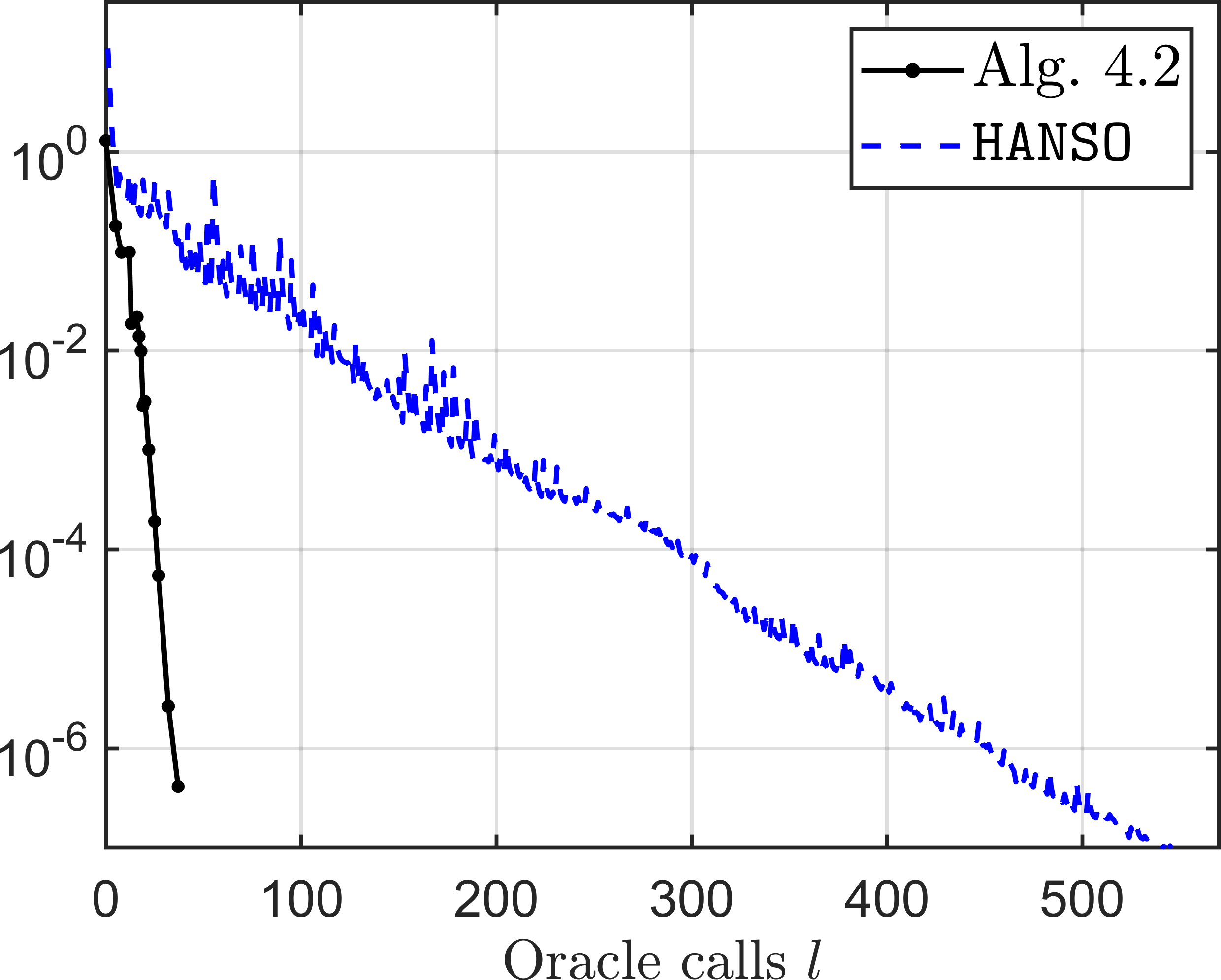}\\
                    (b)
        		}
                \caption{(a) The distance $\| x^j - \tilde{x}^* \|$ in Ex.\ \ref{example:LW2019_eigval} and the corresponding sequence $(\eps_j)_j$. (b) The distance of the objective values to the reference value $f(\tilde{x}^*)$ with respect to oracle calls for \algLocal{} and \hanso{}.}
                \label{fig:example_LW2019_eigval}
            \end{figure}
            Note that the expected convergence behavior can be observed despite the initial $x^1$ not being contained in $\Bcl_{\eps_1}(\tilde{x}^*)$. Fig.\ \ref{fig:example_LW2019_eigval}(b) shows, in black, the distance $f(x^{j(l)}) - f(\tilde{x}^*)$ for $j(l)$ as in Cor.\ \ref{cor:N_step_convergence}. Despite Lem.\ \ref{lem:algo_approx_W_termination}(b) not being applicable, we still observe fast convergence in terms of oracle calls, with \algApproxW{} needing at most $5$ oracle calls in every iteration of \algLocal{}. The blue, dashed line shows the objective value for all oracle calls performed by \hanso{} (cut off at $l = 570$ for better comparison). \hanso{} required $1006$ oracle calls to obtain its final point $\tilde{x}^*$. \algLocal{} required $37$ oracle calls for the final point $x^{15}$, which satisfies $f(x^{15}) - f(\tilde{x}^*) \approx 4.1 \cdot 10^{-7}$. To reach a point with an objective value less than $f(x^{15})$, \hanso{} required $475$ oracle calls. In terms of runtime, \algLocal{} required $2.61$s and \hanso{} required $2.87$s. So while \algLocal{} needed far fewer oracle calls than \hanso{}, the time required for solving the subproblem \eqref{eq:bar_z_epigraph} evens out the comparison here. Furthermore, one should keep in mind that \algLocal{} requires the Hessian matrix for its oracle (when $q \geq 2$), whereas \hanso{} only requires the objective value and the gradient.
        \end{example}

        For the final two experiments, we compare \algLocal{} to other superlinear solvers for nonsmooth optimization problems. The first one is \vubundle{}, which is a Julia implementation of the $\mathcal{V}\mathcal{U}$-bundle method from \cite{MS2005}, and only requires objective values and (sub)gradients as oracle information. As a test problem, we use the convex function from \cite{LO2008}, Sec.\ 5.5 (named \emph{Half-and-half} in \cite{MS2012}), which is again not a finite max-type function.

        \begin{example} \label{example:halfhalf}
            Consider the convex function
            \begin{equation*}
                \begin{aligned}
                    &f : \R^8 \rightarrow \R, \quad x \mapsto \sqrt{x^\top A x} + x^\top B x, \\
                    &A_{i,j} := 
                    \begin{cases}
                        1, & i = j \in \{1,3,5,7\},\\
                        0, & \text{otherwise,}
                    \end{cases}, \quad 
                    B_{i,j} := 
                    \begin{cases}
                        1/i^2, & i = j,\\
                        0, & \text{otherwise.}
                    \end{cases}
                \end{aligned}
            \end{equation*}
            The global minimum is $x^* = 0 \in \R^8$, around which $f$ grows with order $p = 2$. We apply \algLocal{} with $q = 2$, $x^1 = (20.08, \dots, 20.08) \in \R^8$ (as in \cite{MS2012}, p.\ 298), $\eps_1 = 30$, and $\epsthr = 10^{-3}$. The subproblem \eqref{eq:bar_z_epigraph} is solved via \ipopt{}. For \vubundle{} we use the default parameters and the same starting point $x^1$. Fig.\ \ref{fig:example_halfhalf}(a) shows the distance $\| x^j - x^* \|$ for the sequence $(x^j)_j$ generated by \algLocal{}.
            \begin{figure}
                \centering
                \parbox[b]{0.49\textwidth}{
                    \centering 
                    \includegraphics[width=0.45\textwidth]{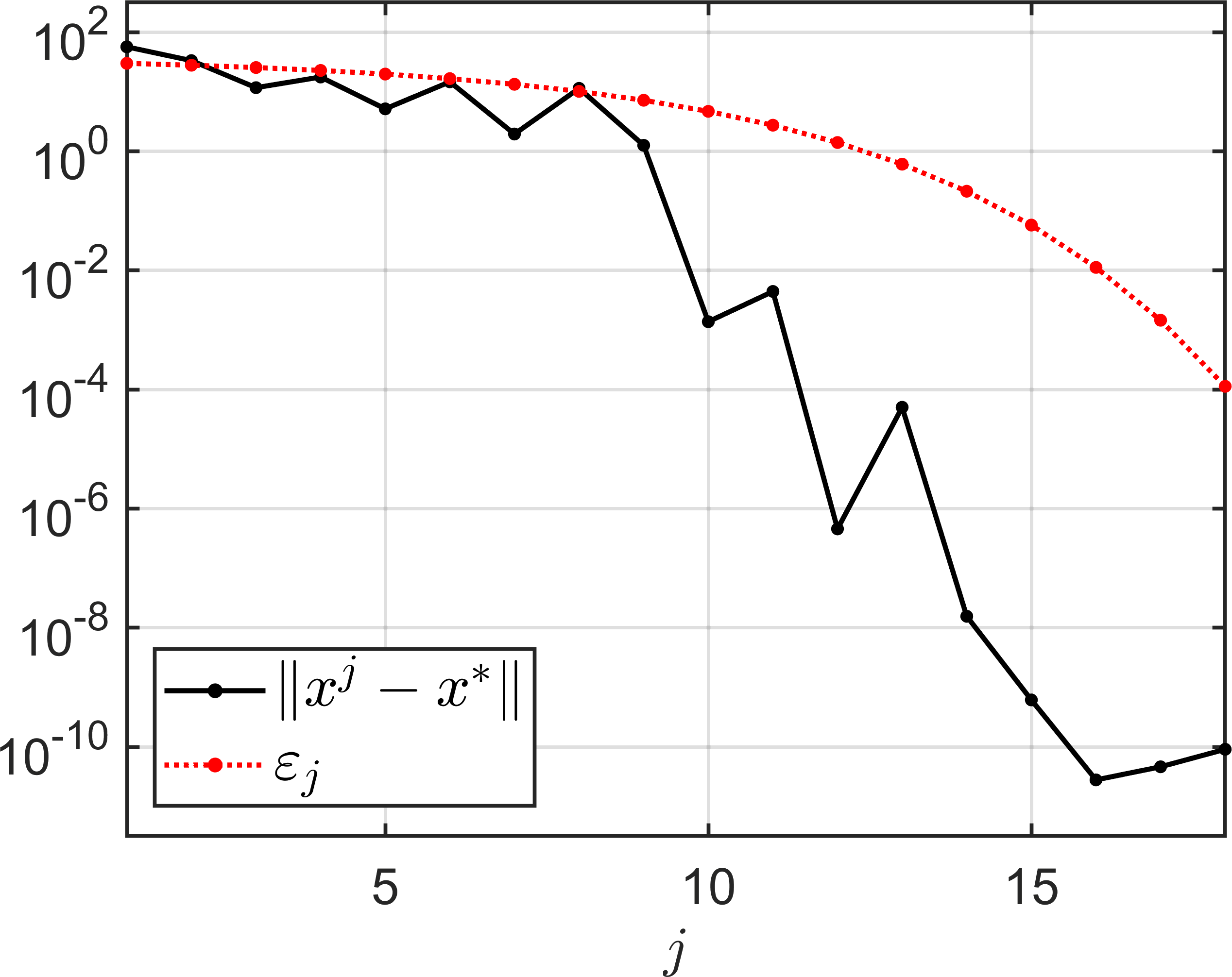}\\
                    (a)
        		}
                \parbox[b]{0.49\textwidth}{
                    \centering 
                    \includegraphics[width=0.45\textwidth]{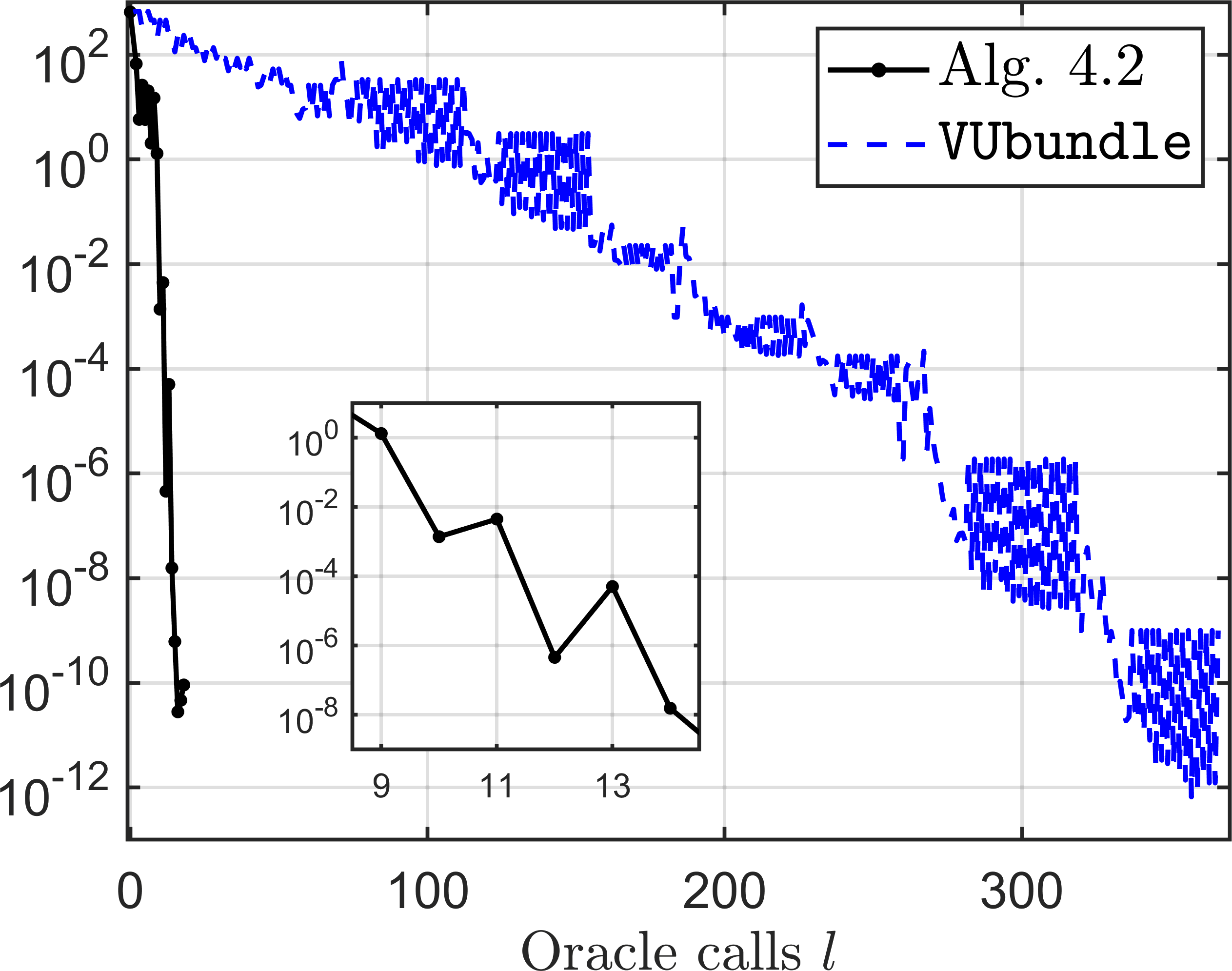}\\
                    (b)
        		}
                \caption{(a) The distance $\| x^j - x^* \|$ in Ex.\ \ref{example:halfhalf} and the corresponding sequence $(\eps_j)_j$. (b) The distance of the objective values to the optimal value $f(x^*)$ with respect to oracle calls for \algLocal{} and \vubundle{}. (The zoom on the result of \algLocal{} shows that it is not a descent method.)}
                \label{fig:example_halfhalf}
            \end{figure}
            We see (roughly) R-superlinear convergence despite multiple of the early iterates $x^j$ not being contained in the corresponding $\Bcl_{\eps_j}(x^*)$. (The lack of improvement once the distance lies below $10^{-10}$ is due to the accuracy of \ipopt{}.) Fig.\ \ref{fig:example_halfhalf}(b) shows, in black, the distance $f(x^{j(l)}) - f(x^*)$ for $j(l)$ as in Cor.\ \ref{cor:N_step_convergence}. (Due to the bundling in \algApproxW{} (cf.\ Rem.\ \ref{rem:bundle_initialization}(a)), every iteration of \algLocal{} only requires a single oracle call for this example.) Since the numbers of oracle calls for objective values and for gradients in \vubundle{} are not equal, and since gradients are typically more costly to compute than objective values, we use the number of gradient evaluations for \vubundle{} in our comparison, shown as the blue, dotted line. While it suggests that for this function, \algLocal{} is more efficient than \vubundle{} in terms of oracle calls, the fact that oracle calls are cheap means that the overall runtime is far slower, with our implementation of \algLocal{} requiring $0.58$s and \vubundle{} only requiring $0.016$s.
        \end{example}

        For the second comparison, we consider the Julia implementation of \superpolyak{} from \cite{CD2024}, which converges superlinearly for functions with a sharp minimum (i.e., \ref{assum:A2} holds with $p = 1$). It requires the objective values and (sub)gradients as oracle information. Additionally, it requires the optimal value $f(x^*)$ to be known. As a test problem, we again consider the nonconvex function from Ex.\ \ref{example:1d_symbolic}. Since $p = 1$ in this case, it suffices to choose $q = 1$ for \algLocal{}, which means that the model is a standard cutting-plane model. By choosing the maximum norm for the trust region (and omitting the exponent $2$ in the constraint), the subproblem \eqref{eq:bar_z_epigraph} then becomes a linear problem, which we can solve using Matlab's \texttt{linprog} solver. 

        \begin{example} \label{example:max_root}
            Consider the function $f$ from Ex.\ \ref{example:1d_symbolic} with $n = 100$. We apply \algLocal{} with $q = 1$, $x^1 = (0.001,0.002,\dots,0.1)^\top \in \R^{100}$, $\eps_1 = 0.5$, and $\epsthr = 10^{-7}$. The subproblem \eqref{eq:bar_z_epigraph} is solved as discussed above. For \superpolyak{} we use the default parameters and the same starting point $x^1$. Fig.\ \ref{fig:example_max_root}(a) shows the distance $\| x^j - x^* \|$ for the sequence $(x^j)_j$ generated by \algLocal{}, suggesting R-superlinear convergence (despite again violating the first inequality in \eqref{eq:local_method_convergence}).
            \begin{figure}
                \centering
                \parbox[b]{0.49\textwidth}{
                    \centering 
                    \includegraphics[width=0.45\textwidth]{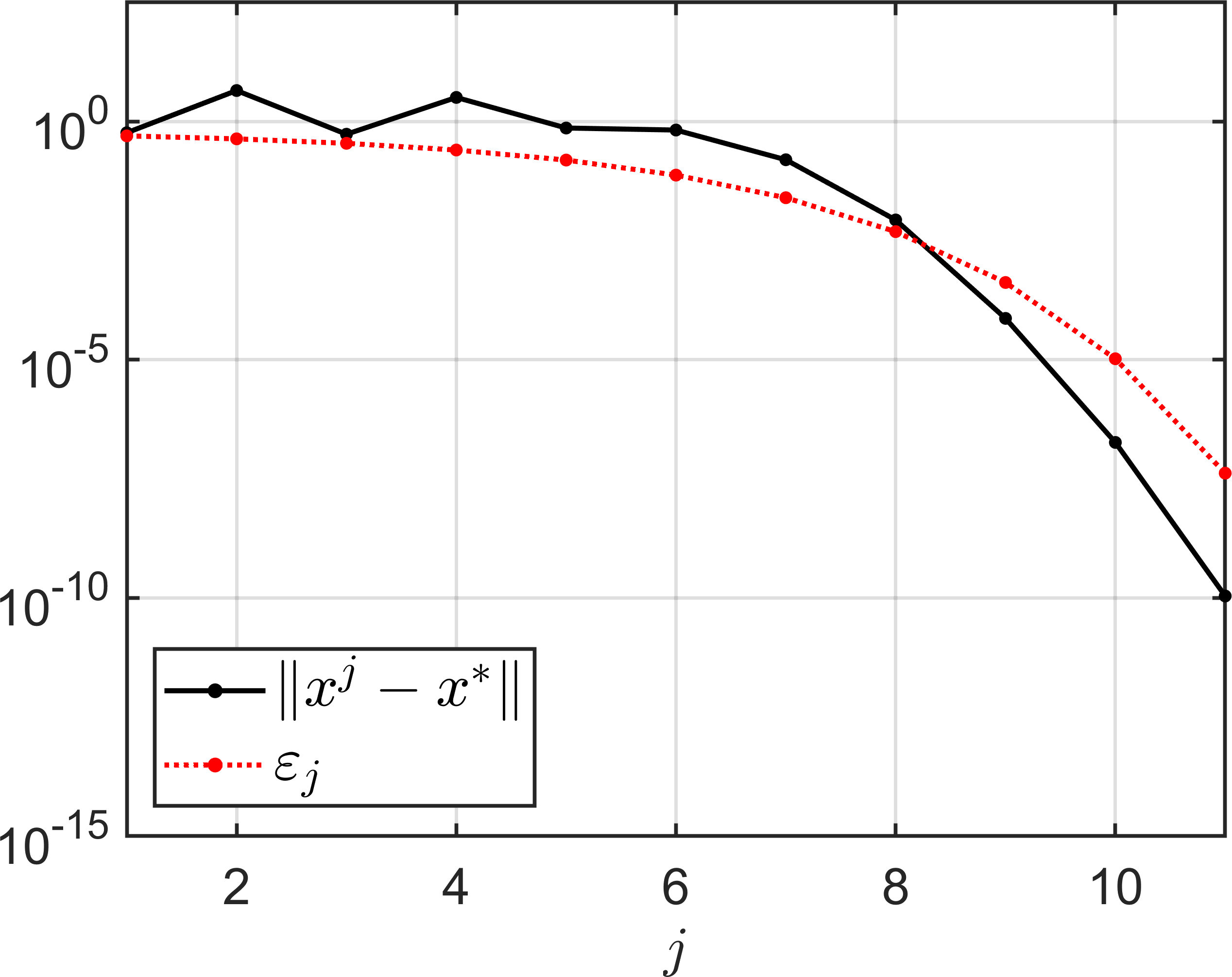}\\
                    (a)
        		}
                \parbox[b]{0.49\textwidth}{
                    \centering 
                    \includegraphics[width=0.45\textwidth]{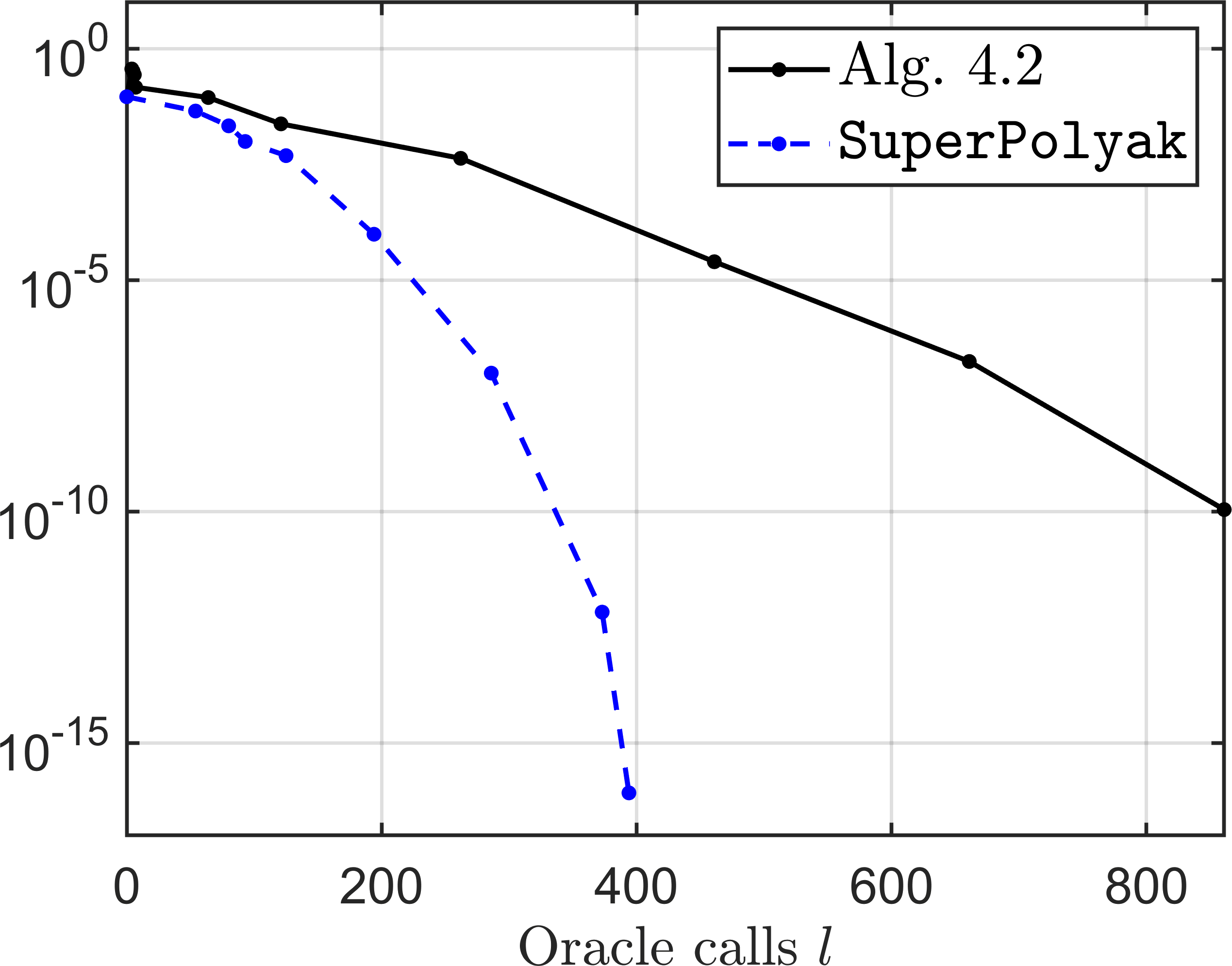}\\
                    (b)
        		}
                \caption{(a) The distance $\| x^j - x^* \|$ in Ex.\ \ref{example:max_root} and the corresponding sequence $(\eps_j)_j$. (b) The distance of the objective values to the optimal value $f(x^*)$ with respect to oracle calls for \algLocal{} and \superpolyak{}.}
                \label{fig:example_max_root}
            \end{figure}
            Fig.\ \ref{fig:example_max_root}(b) shows, in black, the distance $f(x^{j(l)}) - f(x^*)$ for $j(l)$ as in Cor.\ \ref{cor:N_step_convergence}. For $j \geq 9$, \algApproxW{} required the full $|S| = 2n = 200$ iterations (cf.\ Lem.\ \ref{lem:algo_approx_W_termination}(b)), which explains the relatively slow convergence of $(x^{j(l)})_l$. The blue, dotted line shows the number of oracle calls for \superpolyak{}. We see that \superpolyak{} converges significantly faster than \algLocal{}, only needing about half the number of oracle calls. The difference in runtime is even larger, with \algLocal{} requiring $3.39$s (due to the many iterations of \algApproxW{}) and \superpolyak{} only $0.0028$s. 
        \end{example}

\section{Discussion and outlook} \label{sec:discussion_and_outlook}

    We defined higher-order cutting-plane models for \lc{2} functions and showed how they can be used to construct a trust-region bundle method with local R-superlinear convergence. There are multiple directions for future research:
    \begin{itemize}
        \item In the numerical experiments, we used the derivatives of $f$ itself as our oracle information. For $q = 1$, this yields an oracle as in Oracle \ref{oracle:1} (cf.\ the discussion at the beginning of Sec.\ \ref{sec:higher_order_cutting_plane_models}). However, for $q \geq 2$, this may not be the case. For example, for $f$ in Ex.\ \ref{example:halfhalf}, the Hessian matrix $\nabla^2 f(x)$ is unbounded for $x \rightarrow 0 \in \R^n$. Since the selection functions in Def.\ \ref{def:lower_Ck} are continuous in $(s,x)$, there cannot be a representation with $\nabla^2 f_{s(x)}(x) = \nabla^2 f(x)$ for infinitely many $x$ arbitrarily close to $0$. Thus, the practical oracle differs from Oracle \ref{oracle:1} for this function. Nonetheless, this was no issue in any of our numerical experiments. Resolving this gap from theory to practice likely requires more theoretical analysis of the meaning of derivatives of selection functions in local representations of \lc{2} functions. For example, we expect that for finite max-type functions, it is possible to show that there is a local representation of $f$ for which the practical oracle equals the theoretical oracle. Analyzing how well \lc{2} functions can be approximated by finite max-type functions may then close the above gap.
        \item The numerical experiments showed that if functions evaluations are cheap, then in terms of runtime, the current implementation of \algLocal{} is significantly slower than other solvers (on their respective problem classes). In Ex.\ \ref{example:max_root}, roughly $95\%$ of the runtime is taken up by the solution of the subproblem \eqref{eq:bar_z_epigraph}, so a faster implementation can only be obtained by employing a different approach for solving this subproblem. For $q = 2$, it might be possible to exploit the fact that \eqref{eq:bar_z_epigraph} is a \emph{quadratically constrained quadratic program} (QCQP), for which specialized solvers exist (see, e.g., \cite{L2005}). Alternatively, one could consider approximate solutions, since intuitively, exact solutions should only be necessary ``in the limit'' as the sequence approaches the minimum. For example, it may be possible to use ideas from SQP methods to first estimate the Lagrange multipliers in \eqref{eq:bar_z_epigraph} and then compute an approximate solution based on the estimated multipliers, similar to the approach of \cite{LW2019}, Sec.\ 3.6. 
        \item Our convergence theory for \algLocal{} technically requires global solutions of the subproblem \eqref{eq:bar_z_epigraph} (for the inequality \eqref{eq:proof_lem_T_W_error_estimate} in the proof of Lem.\ \ref{lem:T_W_error_estimate}). While we did not encounter any convergence issues in our numerical experiments when using non-global solvers, the effect of local solutions of \eqref{eq:bar_z_epigraph} on the convergence still has to be properly analyzed.
        \item Clearly, derivatives of an order larger than $1$ may be cumbersome to provide and to work with. Considering the derivation of quasi-Newton methods from Newton's method in smooth optimization, it should be analyzed whether there is a quasi-Newton version of \algLocal{} that only requires first-order information and approximates the Hessian matrix.
        \item We have only presented one way to overcome the Challenges \ref{enum:C1}, \ref{enum:C2}, and \ref{enum:C3}. Other approaches may lead to different superlinearly convergent methods, which could be superior to \algLocal{}.
    \end{itemize}

\vspace{10pt}

\setlength{\bibsep}{0pt plus 0.3ex}

\noindent \textbf{Acknowledgements.} \quad This research was funded by Deutsche Forschungsgemeinschaft (DFG, German Research Foundation) – Projektnummer 545166481.

\bibliography{references}

\end{document}